\documentclass{article} 
\usepackage{nips13submit_e,times}

\usepackage{booktabs} 

\usepackage{graphicx}
\usepackage{lipsum}
\graphicspath{ {images/} }

\usepackage{xcolor}

\usepackage{amsthm}
\usepackage{tikz}
\usepackage{pgfplots}

\usepackage{amsmath}
\usepackage{amssymb}
\usepackage{amsthm}

\usepackage[linesnumbered,ruled]{algorithm2e}

\newcommand{\Bbl}{\Big (}
\newcommand{\Bbr}{\Big )}
\newcommand{\Bsbl}{\Big [}
\newcommand{\Bsbr}{\Big ]}

\newtheorem{definition}{Definition}
\newtheorem{lemma}{Lemma}
\newtheorem{assumption}{Assumption}
\newtheorem{proposition}{Proposition}
\newtheorem{corollary}{Corollary}

\newcommand{\union}{\bigcup\limits}

\newcommand{\setB}{{\mathcal B}}

\newcommand{\setA}{{\mathcal A}}
\newcommand{\setR}{{\mathcal R}}
\newcommand{\setT}{{\mathcal T}}
\newcommand{\setC}{{\mathcal C}}
\newcommand{\setD}{\hat {\mathcal C}}
\newcommand{\setN}{{\mathcal N}}
\newcommand{\setp}{{\mathcal P}}
\newcommand{\setS}{{\mathcal S}}
\newcommand{\setSone}{{\mathcal S}_p}

\newcommand{\setP}{{\mathcal P}(\optbg)}

\newcommand{\setCp}{\setC(p^*,\optI)}
\newcommand{\setCpc}{\setC_c(p^*,\optI)}
\newcommand{\setCpt}{\setC_t(p^*,\optI)}

\newcommand{\setCq}{\setC(q^*,\optI)}

\newcommand{\setCqt}{\setC_t(q^*,\optI)}

\newcommand{\setCap}{\setC(p,\alloc)}
\newcommand{\setCapc}{\setC_c(p,\alloc)}

\newcommand{\setCcone}{{\setC}_{c,1}}
\newcommand{\setCctwo}{{\setC}_{c,2}}
\newcommand{\setCtone}{{\setC}_{t,1}}
\newcommand{\setCttwo}{{\setC}_{t,2}}

\newcommand{\settCc}{{\tilde \setC}_{c}}
\newcommand{\settCt}{{\tilde \setC}_{t}}

\newcommand{\settCcone}{{\tilde \setC}_{c,1}}
\newcommand{\settCctwo}{{\tilde \setC}_{c,2}}
\newcommand{\settCtone}{{\tilde \setC}_{t,1}}
\newcommand{\settCttwo}{{\tilde \setC}_{t,2}}

\newcommand{\setDp}{(\setDpc,\setDpt,p,\alloct)}
\newcommand{\setDpc}{\setD_{c}}
\newcommand{\setDpt}{\setD_{t}}

\newcommand{\setDq}{(\setDqc,\setDqt,q,\alloct)}
\newcommand{\setDqc}{\setD_{c}}
\newcommand{\setDqt}{\setD_{t}}

\newcommand{\setDqk}{(\setDqc^{(k)},\setDqt^{(k)},q^{(k)},\alloct^{(k)})}
\newcommand{\setDqck}{\setD^{(k)}_{c}}
\newcommand{\setDqtk}{\setD^{(k)}_{t}}

\newcommand{\setDql}{(\setDqc^{(l)},\setDqt^{(l)},q^{(l)},\alloct^{(l)})}
\newcommand{\setDqcl}{\setD^{(l)}_{c}}
\newcommand{\setDqtl}{\setD^{(l)}_{t}}

\newcommand{\setDqcone}{\setD^{(1)}_{c}}
\newcommand{\setDqtone}{\setD^{(1)}_{t}}

\newcommand{\alloc}{\allocb,\allocg}
\newcommand{\allocbg}{\allocb,\allocg}
\newcommand{\allocb}{\boldsymbol{b}}
\newcommand{\allocg}{\boldsymbol{\gamma}}

\newcommand{\opt}{\optb,\optg}
\newcommand{\optbg}{\opt}
\newcommand{\optb}{\boldsymbol{b}^*}
\newcommand{\optg}{\boldsymbol{\gamma}^*}
\newcommand{\optI}{\optbg}
\newcommand{\optt}{\boldsymbol{t}^*}

\newcommand{\alloct}{\boldsymbol{t}}

\begin{document}
\title{Optimal Bidding Strategies for Online Ad Auctions with Overlapping Targeting Criteria}
\author{%
  Erik Tillberg\\
  Department of Computer Science\\
  University of Toronto
  \And
  Peter Marbach \\
  Department of Computer Science \\
  University of Toronto \\
  \AND
  Ravi Mazumder \\
  Department of Electrical and Computer Engineering\\
  University of Waterloo}


\nipsfinalcopy 
\maketitle


\begin{abstract}
We analyze the problem of how to optimally bid for ad spaces in online ad auctions. For this we  consider  the general case of multiple ad campaigns  with overlapping targeting criteria.  In our analysis we first characterize the structure of an optimal bidding strategy. In particular, we show that an optimal bidding strategies decomposes the problem into disjoint sets of campaigns and targeting groups. In addition, we show that pure bidding strategies that use only a single bid value for each campaign are not optimal when the supply curves are not continuous. For this case, we derive a lower-bound on the optimal cost of any bidding strategy, as well as mixed bidding strategies that either achieve the lower-bound, or can get arbitrarily close to it. 


\end{abstract}


\section{Introduction}\label{sec:intro}
Online advertising has grown dramatically in recent years and is by now a massive market. To service the demand, ad exchanges have been set up in several places such as  New York and Toronto. Ad exchanges allow advertisers to target their ad campaigns towards  specific demographic groups. This leads to a new set of problems for bidding on ad spaces where specific targeting criteria need to be taken into account.  These decision problems are challenging due the large number of ad spaces that can be used to potentially place a given ad, as well as the competition between ad campaigns that target the same ad spaces. As a result, how to optimally bid for ad spaces has been an open problem, and commercial platforms bidding for ad spaces have been relying on  heuristics to make these decisions. 

In this paper we address this problem.
More precisely, we consider the situation where rather than directly buying ad spaces, advertisers use a broker, or a co-called demand supply platform (DSP), that acquires ad spaces on their behalf. The interaction between a DSP and advertiser is governed by a contract under which the DSP is obliged to acquire a certain number of ad spaces over a given period of time, where each ad space satisfies the targeting criteria. We refer to this contract as an ad campaign.
We assume that a DSP represents one or more campaign of one or more advertiser. Each campaign has its own targeting criteria, as well as a required number of ad spaces that the DSP needs to obtain for the campaign. 

A transaction (opportunity to buy an ad space) in an online ad marked unfolds as follows. The transactions begins with an end user visiting a web page or mobile app. This visit triggers a bid request that is sent to the ad exchange. The bid request contains information about the ad space, as well as about the end user that has accessed the web page or mobile ad. The ad exchange then forwards the bid request to the demand side platforms (DSPs).
A DSP that receives the bid request can then decide  whether or not to bid for the ad space by submitting a bid for the ad space. The ad exchange collects all the bids submitted by DSPs for the ad space, and uses a second-price auction to decide which of the competing DSPs will obtain the ad space.
The DSP that wins the auction can display an ad  to the end user on that ad space. Once an ad has been displayed, and can be viewed by the end user, this ad space is denoted, and counts, as an impression.

In this paper we analyze how a DSP should bid on a bid requests from the ad-exchange in order to obtain the required number of impressions for the ad campaigns it represents at the lowest possible cost.
We formally model this decision problem as an optimization (cost minimization) problem. To do that, we introduce the concept of targeting groups that groups bid requests from ad campaigns based on the bid request types that match the targeting criteria of ad campaigns. In addition, we associate with each targeting group a  supply curve that characterizes the number (volume) of impressions the DSP obtains during a given decision as a function of the bid value.

The main contributions of the paper are:
\begin{enumerate}
\item[a)] we provide a formal analysis of how a DSP optimal bids for ad spaces  in an online ad market for the case where a DSP has several campaigns, and the campaigns can have overlapping targeting criteria. In addition, we do not assume that supply curves to be convex and continuous functions, but only assume them to be  right-continuous and non-decreasing functions.
\item[b)] we show that an optimal bidding strategies decomposes the problem into disjoint sets of campaigns and targeting groups, where an optimal bidding strategy can be computed for each component individually. 
\item[c)] we show that pure bidding strategies that use only a single bid value for each campaign and targeting group pair are not optimal when the supply curves are not continuous. For this case we derive a lower-bound on the optimal cost of any bidding strategy, and provide  mixed bidding strategies that either achieve the lower-bound, or can get arbitrarily close to it.  
\end{enumerate}
We note that discontinuous supply curves are the norm in practical applications (see Fig.~\ref{fig:supply_curves} for an illustration). As a result, the problem considered in this paper can not be solved using  standard techniques from optimization theory to compute an optimal bidding strategy, which makes the analysis challenging. In the paper we provide a solution approach and technique that can be used to analyze supply curves that are right-continuous, non-decreasing functions, and thus not necessarily continuous or convex functions.

The rest of the paper is organized as follows. In Section~\ref{sec:related} we provide an overview of related work. In Section~\ref{sec:auction} we describe the second price auction that is used in online ad markets.
In Section~\ref{sec:model} we present the optimization problem to that we use to model the decision problem of the DSP. In Section~\ref{sec:single_target} we consider the special case of a single campaign with a single targeting group in order to illustrate the main intuition behind our analysis and results. In Section~\ref{sec:opt_solution}, we show that when the supply curves are given by continuous functions, then an optimal bidding strategy decomposes the problem into disjoint sets of campaigns and targeting groups,  where an optimal bidding strategy can be computed for each component individually.
In Section~\ref{sec:algorithm} we present the algorithm to compute an bidding strategy, and show in Section~\ref{sec:results} that this bidding strategy is optimal when the supply curves are given by continuous functions. In addition, we derive in Section~\ref{sec:results} a lower-bound on the optimal cost of any bidding strategy, and provide  mixed bidding strategies that either achieve the lower-bound, or can get arbitrarily close to it. In Section~\ref{sec:implementation} we discuss how the results obtained in the paper can be applied to practical systems. 



\section{Related Work}\label{sec:related}

In this section we provide an overview of the literature related to optimal bidding algorithms for DSPs. We first consider the literature that is directly related to the problem consider in this paper, i.e. existing literature on optimal bidding algorithms for DSPs.
Roughly, the existing literature on the DSP bidding problem can be grouped into the following four categories: 1) single campaigns,  2) many campaigns with identical targeting criteria, 3) many campaigns with different but non-overlapping targeting criteria, and 4) many campaigns with overlapping targeting criteria. In this paper we consider the last case, i.e. the general case where a DSP represents many campaigns that  have overlapping targeting criteria.

{\it The single campaign case:} There exist several results for the basic case of optimal bidding algorithms where a DSP respresents a single ad campaign~\cite{jiang,gummadi,zhang2014}, and  the goal of the DSP is to maximize the utility (valuation) of the obtained ad spaces given a budget constraint.
Compared with our analysis, this literature  does not consider the case where the supply curves are not necessarily differentiable, or continuous.



{\it Many Campaigns with Identical Targeting Criteria: } The first generalization of the single campaign case is to consider the case where a DSP respresents many campaigns, but all campaigns have an identical targeting criteria. This case has been studied in~\cite{zhang2015} where the authors propose a bidding strategy that maximizes revenue subject to a global budget and risk constraint; risk being the variance of revenue. The resulting algorithm decides on the optimal bid value, as well as the probabilities with which bid requests are assigned to a given campaign. This work considers a different problem (with a budget and risk constraint rather than a constraint on the number of impressions that a campaign needs to obtain) than the one considered in this paper, and the bidding strategies of~\cite{zhang2015} can not be applied, and would not be optimal, for the problem considered in this paper.

{\it Many Campaigns with Non-Overlapping Targeting Criteria:} In~\cite{zhang2012}, the authors consider the case where a DSP represents many campaigns and the campaigns have different, but disjoint, targeting criteria. For this case, ~\cite{zhang2012} considers the problem of maximizing the total utility (valuation) over all campaigns subject to a global budget constraint. The decision variables of the corresponding optimization problem are the amount of total budget that is allocated to each campaign, and the price that is used for to bid for ad spaces for the individual campaigns. The work in~\cite{zhang2012} derives an optimal bidding strategy for this case assuming that the supply curves (distribution function of the winning price for a given targeting criteria) are convex. This allows the problem to be formulated as a convex optimization problem, and can be solved using standard techniques from convex optimization. In addition, the analysis in~\cite{zhang2012} does not provide the structural properties of an optimal solution that we present in Section~\ref{sec:results} that shows that an optimal allocation decomposes the problem into disjoint components that can be solved individually, and that pure strategies are not optimal when the supply curves are not continuous (which is the case for the supply curves obtained in practice).

{\it Many Campaigns with Overlapping Targeting Criteria:} The most general case, and the case that we consider in this paper, is where a DSP represents many campaigns that can have overlapping targeting criteria. Obtaining an optimal bidding strategy for this case has been an open problem, and so far only heuristic strategies  have been considered  for this case~\cite{perlich,spentzouris}. 
In~\cite{perlich} the authors propose a heuristic algorithm to compute bidding values, and allocate bid requests to campaign, for the case where the DSP aims at maximizing the number of clicks that are received under a given budget constraint. The proposed algorithm uses a linear (heuristic) bidding function that is proportional to the predicted click through rate of a campaign for a given bid request, and allocates the bid request to the campaign with the highest bid value. In essence, the proposed algorithm allocates bid requests to the campaign that has the highest click-through rate. In~\cite{spentzouris}, the authors consider the special case where two campaigns either have identical targeting criteria (fully overlap), or have disjoint criteria (do not overlap). For this case, the authors propose a heuristic algorithm that uses linear regression to forecast the cost and number of clicks received. This forecast function is then used to formulate the decision problem of a DSP as an optimization problem that can be solved.

A large fraction of the literature on optimal bidding strategies for DSPs considers the case where a DSP aims at  acquiring a many ad spaces as possible under a given budget constraint, or at maximizing a the utility of the obtained ad spaces under a given budget constraint. In this paper, we consider a slightly different objective where the DSP needs to acquired a predefined number of ad spaces for each campaign, and wants to achieve this at the lowest possible cost.  The latter objective reflects that scenario that many DSPs face in practice.

Next  we highlight existing literature that is related to the general problem of online ad markets, but considers a fundamentally different problem compared with the one consider in this paper.

In~\cite{balseiro}, the authors study the interaction between several advertisers/DSPs  that are competing for an identical set of bid requests, i.e. each advertiser/DSP has  single campaign and the  campaigns (from the different advertisers/DSPs) have all the same targeting criteria. Using the game-theoretic approach, the authors derive a dynamic bidding strategy (for each advertiser/DSP) and show that the bid values under this strategy converges to a Nash equilibrium.

In~\cite{balseiro2014}, the authors study online ad markets from the perspective of the ad publishers (owners of the web pages or mobile ads where where the as spaces are located), rather than the DSP. In particular, \cite{balseiro2014} considers the situation where a publishers can choose between selling ad spaces either on contract-based market, or on a spot market (ad exchange).
A key difference between this problem, and the problem considered in this paper,  is that the total number of ad spaces that are available (at the publisher) are given (as part of the problem formulation), and do not depend on a decision variable (bid price) of the optimization problem.

There is a large body of existing literature on the problem
where a centralized platform collects available ad spaces, as well as the preferences of advertisers regarding the ad spaces, and the objective of the platform is to match ad spaces with advertisers with respect to a given criteria  (see~\cite{mehta2013} for an overview). This problem corresponds to  a  bi-partite matching problems that is NP (see discussion in~\cite{mehta2013}). As a result this body of research focuses on obtaining  efficient approximate algorithm for the corresponding bi-partite matching problems.  We note that this problem is fundamentally different from the problem considered in this paper.


\section{Second Price Auction}\label{sec:auction}
Before we define in Section~\ref{sec:model} the model that we use to study the optimal decisions by a DSP in an online ad market, we describe in this section the second price, sealed bid auction, that is used by  ad exchanges use to decide which DSP wins the ad spot for a given bid request. 

Suppose that there are $N$ DSPs submitting a bid for a given bid request, and let $(b_1,...,b_N)$ be the bids that are received by the ad exchange and assume that
$$b_i \neq b_j, \qquad i\neq j, i,j =1,...,N.$$
Then the DSP $i^* \in \{1,...,N\}$ with the highest bid $b_{i^*}$, i.e. we have that
$$ b_{i^*} \geq b_i, \qquad i=1,...,N, $$
will win the auction and receive the ad spot of the bid request. The price that DSP $i^*$ pays for the ad spot is equal to the value of the second highest bid in $(b_1,...,b_N)$, i.e. the price $p$ is equal to $b_s  \in \{1,...,N\}$ such that
$$ b_s \geq b_i, \qquad i \in \{1,...,N\}\backslash\{i^*\}.$$
We ignore in this paper the floor price of an auction; a floor price can easily be introduced in our analysis without changing the results.

Whenever a DSP wins an ad space for a campaign that it represents, we say that the DSP has received an impression (ad space) for the campaign.


\section{Problem Definition}\label{sec:model}

Recall that we consider the problem of how a DSP should bid for ad spaces in an online ad market. We assume that the DSP represents one or more ad campaigns. Each campaign is given by a contract that specifies the targeting criteria for the campaign, and the number of impressions the campaign needs to obtain over a given time period (for example, a month). Below we describe the model that we use to analyze this problem. 

\subsection{Decision Period}
We focus in our analysis on the case where the DPS is required to spread out the total number of impression over the given time period, and needs to obtain an equal number of impressions in during a number of pre-specified smaller time intervals (for example, a day).  In addition, we allow the DSP to break the time intervals over which a required number of impressions needs to be obtained into smaller decision periods. A the length of a decision period could relatively long such as an hour, or  shorter such as a few minutes. In the following, we analyze how a DSP optimally decides on the bid values and bid request allocations for a given decision period. To do that, we formally model the decision problem as an optimization problem

\subsection{Campaigns}
Let $\setN = \{1,...,N_c \}$
be the set of campaigns that the DSP represents. For each campaign $i \in \setN$, let $I_i$ be the total number of impressions that campaign is required (scheduled) to obtain during the given decision period.

We assume that each bid request $r$ that the DSP receives from the ad exchange,  is characterized by its properties $t(r)$. Examples of a bid request properties  could be the location, age, and gender,  of the end user who accessed the web page or mobile app and triggered the bid request.  We refer to $t(r)$ as the type of the bid request. The type $t(r)$ determines whether  bid request $r$ matches the targeting criteria of a given campaign $i \in \setN$. More precisely, let $\setR$ be the set of all possible bid request types, and let $\setR_i \subseteq \setR$ be  the set of all bid request types that match the targeting criteria of campaign $i$. For example, the targeting criteria of an a campaign could be to  show ads to  users located in New York City aged between 30-50 years.

If for a given bid request $r$ matches the targeting criteria of  campaign $i$ and we have that $t(r) \in \setR_i$, then we say that $r$ is an admissible bid request for campaign $i$.

\subsection{Targeting Groups}
The targeting criteria $\setR_i$ of the campaigns $i \in \setN$ are not necessarily disjoint sets. Moreover, the targeting criteria $\setR_i$ and $\setR_j$ of two ad campaigns $i$ and $j$ that overlap, may not fully overlap. As a result the set of targeting criteria $\{\setR_i\}_{i \in \setN}$ has  very little structure, which makes it difficult to use them directly to formulate an optimization problem for  the decision problem of the DSP. To overcome this problem, instead of using the targeting criteria directly, we use the set of targeting criteria $\setR_i$, $i \in \setN$, to construct a set $\setT = \{1,...,N_r\}$ of $N_r$ targeting groups with the following structure.

\begin{definition}\label{def:targeting_groups}
A set of targeting groups $\setT$ is given by $N_r$ sets $\setT_j \subset \setR$, such that
$$\bigcup\limits_{j=1}^{N_r} \setT_j = \bigcup\limits_{i=1}^{N_c} \setR_i$$
and 
$$ \setT_m \cap \setT_n = \emptyset, \qquad m \neq n, m,m=1,...,N_r.$$
Furthermore, for every campaign $i$ there exists a set $\setA_i \subseteq \{1,...,N_r\}$ of targeting groups such that
$$\setR_i = \bigcup\limits_{j \in \setA_i} \setT_j$$
and
$$\setR_i \cap \setT_j = \setT_j, \qquad j \in \setA_i.$$
For given a targeting group $j$, the set $\setB_j$ is given by
$$ \setT_j \cap \setR_i = \setT_j.$$
\end{definition}
Note that as the number of campaigns is finite, a set of targeting groups as given by Definition~\ref{def:targeting_groups} can always be constructed. Below we discuss properties of targeting groups as given by Definition~\ref{def:targeting_groups} that play an important role in our analysis.

The set $\setA_i$, $i \in \setN$, in Definition~\ref{def:targeting_groups} consists of all the targeting groups that campaign $i$ can bid on. That is, given a bid request $r$ such that $t(r) \in \setR_i$, we have that there exists a $j \in \setA_i$ such that $t(r) \in \setT_j$.
Furthermore as the targeting groups are given by disjoint sets, for given a bid request $r$ such that $t(r) \in \setR_i$ there exists exactly one  $j \in \setA_i$ such that $t(r) \in \setT_j$.

Similarly, the set $\setB_j$, $j \in \setT$, consists of all of the campaigns that can bid on a request $r$ such that $t(r) \in \setT_j$ and $t_r \in \setR_i$. That is, it consists of all the campaigns that can bid on the set of targets in $\setT_j$. 

The key property of the targeting groups $\setT$ is that  the sets $\setT_j$, $j \in \setT$, are disjoint, and hence each  type $t(r)$ of a bid request $r$ maps to exactly one targeting group.



\begin{figure*}
\centering
\includegraphics[width=0.6\textwidth]{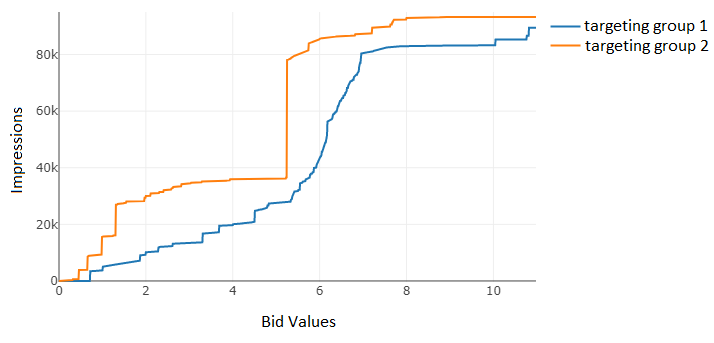}
\caption{Two examples of supply obtained from real-life data.}
\label{fig:supply_curves}
\end{figure*}

\subsection{Supply Curves}\label{sec:supply_curves}
Given the set $\setT$ of targeting groups, we associate with each targeting group $j \in \setT$  a supply curve $D_j(b)$, $b \geq 0$. The supply curve $D_j(b)$ has the following interpretation.  If during the decision period the  DSP bids with a fixed bid value $b_j$ for all bid requests that match targeting group $\setT_j$, $j \in \setT$, then $D_j(b_j)$ is  the total number of impressions that the DSP obtains from targeting group $j$ over  the decision period. Note that by definition, the supply curves  $D_j(b)$, $b\geq 0 $, are non-decreasing function. Fig.~\ref{fig:supply_curves} shows two supply curves that have been obtained from actual data from an online ad exchange. From Fig.~\ref{fig:supply_curves} we have that the supply curves are not necessarily given by differentiable, or continuous, functions, and we  make the following assumption regarding the supply curves  $D_j(b)$, $j \in \setT$.

\begin{assumption}\label{ass:supply_curves}
For each targeting group $j \in \setT$ the supply curve $D_j(b):[0,\infty) \mapsto [0,\infty)$ is a non-decreasing right-continuous function.
\end{assumption}      


Using this definition of a supply curve, we have that if the  DSP bids on all bid requests that match a given  targeting group $j \in \setT$ with a fixed bid $b_j$, then the DSP will obtain $D_j(b_j)$ impressions for  targeting group $j$ with a cost of  (see for example~\cite{gallager})
\begin{eqnarray*}
&& \int_{0}^{b_t} b dD_j(b) 
= D_j(b_j)b_j - \int_0^{b_j}D_j(b) db.
\end{eqnarray*}

\subsection{Decision Variables}

For each targeting group $j \in \setT$ the DSP has to decide on
\begin{enumerate}
\item[a)] the fraction $\gamma_{i,j}$, $i \in \setB_j$, of bid requests from the targeting group $j$ that are allocated to campaign $i$ to bid on, and
\item[b)] the bid value $b_{i,j}$ that it uses for bidding on requests in targeting group $j$, $j \in \setA_i$, that are allocated to campaign $i$.
\end{enumerate}

We then use the following notation for the decision variables of the DSP during a given decision period. 
Let
$$\allocg_j = \{\gamma_{i,j}\}_{i \in \setB_j}, \qquad i=1,...,N_r,$$
be the faction of bid request from targeting group $j$ that are allocated to campaign $i$, and let 
$$\allocb_j = \{b_{i,j}\}_{j \in \setA_i}, \qquad i=1,...,N_c,$$
be the bid values that campaign $i$ uses for bid request from targeting group $j$. 

Furthermore, let
$$\allocb = (\allocb_1, ...,\allocb_{N_c}), \mbox{ and } \allocg = (\allocg_1, ...,\allocg_{N_r}),$$
be the set of decision variables for all campaigns, and all targeting groups, and let $(\alloc)$ denote the corresponding allocation of the DSP for a given decision period.

If the DSP allocates for a given decision a fraction $\gamma_{i,j}$ of the bid requests of targeting group $j$ to campaign $i$, then the total number of impressions that campaign $i$ will obtain from targeting group $j$ is given by $\gamma_{i,j} D_j(b_{i,j})$, and the total cost for obtaining  these impressions is given by
$$ \gamma_{i,j} \Bsbl D_j(b_{i,j}) -   \int_0^{b_{i,j}}D_j(b) db \Bsbr.$$

\subsection{Optimization Problem}
Recall that for a given decision period the goal of the DSP is to obtain the required number of impressions for each campaign $i \in \setN$ at the lowest possible cost.
Using the above notations and definitions, the corresponding optimization problem  is given by
\begin{equation}\label{eq:opt_problem}
\begin{aligned}
& \underset{(\alloc) }{\text{minimize}}
& & \sum_{i=1}^{N_c}\bigg[\sum_{j=1}^{N_r} \gamma_{i,j} \Bbl D_j(b_{i,j})b_{i,j} - \int_0^{b_{i, j}}D_{j}(b)db \Bbr \bigg] \\
& \text{subject to}
  & & \sum_{j \in \setA_i} \gamma_{i,j} D_j(b_{i, j}) \geq I_i, \; i=1,...,N_c, \\
& & & \sum_{i \in \setB_j} \gamma_{i,j} \leq 1, \; j=1,...,N_r, \\
& & & b_{i,j} \geq 0, \; i=1,...,N_c, j=1,...,N_r,  \\
  & & & \gamma_{i,j} \geq 0, \; i=1,...,N_c, j=1,...,N_r.
\end{aligned}
\end{equation}
The first inequality requires the number of impressions obtained for each campaign to meet the required number of impressions.  The second inequality states that the total fraction of bid request that are allocated to campaigns can not exceed 1. The last two constraints state that the  the bid values and and allocations can not be negative.

In the following we refer to an allocation $(\alloc)$ that satisfies the above conditions as a solution to the optimization problem  given by Eq.~\eqref{eq:opt_problem}


For our analysis we assume that the  DSP can always get the required impressions by using high enough bid values $b_{ij}$, $i\in \setN, j \in \setT$, 
and thus the optimization problem has a feasible solution.


\subsection{Online Bidding Algorithm}\label{sec:online_algo}
In this section discuss how a  solution $(\alloc)$ of the optimization problem given by Eq.~\eqref{eq:opt_problem} can be used to implement an online bidding algorithm. 

Let $(\alloc)$ be a solution for the optimization problem given by Eq.~\eqref{eq:opt_problem}.  The solution $(\alloc)$ can then be used to implement an online algorotihm as follows. Whenever a new bid request $r$ is received by the DSP, the DSP first determines the type $t(r)$ of the bid request. Next, the DSP determines the targeting group the type $t(r)$ belongs to. Let $j$ be this targeting group. Using the solution $(\alloc)$, the DSP then allocates with  probability $\gamma_{i,j}$ the bid request to campaign $i$, $i \in \setB_j$; and uses the bid value $b_{i,j}$ to obtain the ad space for campaign $i$.

Note that the solution $(\alloc)$ is (pre-)computed in the background (off-line), and the values of the solution $(\alloc)$ are used to implement the online bidding algorithm described above. The length of a decision period determines how often the solution $(\alloc)$ is re-computed, and the bid values and allocations for the online algorithm are updated. In Section~\ref{sec:implementation} we discuss additional implementation issues of the online algorithm.

\section{Single Campaign with a Single Targeting Group}\label{sec:single_target}

Before we start with the analysis of an optimal allocation for the optimization problem given by Eq.~\eqref{eq:opt_problem}, we provide in this section the main intuition behind our analysis and results. For this we consider the special case where a single campaign $i=1$ gets a total of $I_1$ impressions  from a single targeting group $j=1$. We assume  that
$$\lim_{b \to \infty} D_1(b) > I_1,$$
and there is sufficient supply to support the impressions of the campaign.

Our first result characterizes the optimal allocation $(\opt)$  for the case where there exists a bid value $b^*$ such that $D_1(b^*) = I_1$.
\begin{proposition}\label{prop:single_opt_cont}
Consider the special case where a  single campaign $i=1$ gets impressions from a single targeting group $j=1$, and $I_1$ is the total number of impressions that the campaign  is required to obtain.
If there exists a bid value $b^*$ such that
$$D_1(b^*) = I_1,$$
then the allocation $(\opt)$ given by
$$\gamma^*_{1,1} = 1 \quad \mbox{ and } \quad  b^*_{1,1} = b^*$$
is an optimal allocation. Furthermore, the optimal cost is given by
$$I_1 b^* - \int_0^{b^*} D_1(b) db.$$
\end{proposition}  
We provide a proof of Proposition~\ref{prop:single_opt_cont} in Appendix~\ref{app:single_opt_cont}. 
Proposition~\ref{prop:single_opt_cont} implies that if the supply curve $D_1(b)$ is given by a continuous function, then the allocation $(\opt)$  given in  Proposition~\ref{prop:single_opt_cont} is an optimal allocation. To see this, note that if the supply curve $D_1(b)$ is given by a continuous function, then there exists a bid value $b^*$ such that
$$D_1(b^*) = I_1.$$

If the the supply curve $D_1(b)$ is not continuous, and there does not exist a bid value $b^*$ such that $D_1(b^*) = I_1$, then we define the  bid value $b^*$ such that  
\begin{equation}\label{eq:opt_b_1}
  D_1(b^*) \geq I_1
\end{equation}  
and
\begin{equation}\label{eq:opt_b_2}
  D_1(b) < I_1, \quad 0 \leq b < b^*.
\end{equation}
In the next subsection we study optimal allocations for this case. To do that we use the following notation. We  let 
$$ D^{(-)}(b^*) = \lim_{b \uparrow b^*} D_1(b) < D_1(b^*).$$

\subsection{Pure Bidding Strategies}

Note that the optimization problem given by Eq.~\eqref{eq:opt_problem} considers allocations under which each campaign $i \in \setN$ uses (as most) a single bid value $b_{i,j}$ for targeting group $j \in \setA_i$.  We refer such an allocation  as a  pure bidding strategy. Let $\setSone$ be the set of all pure bidding strategies. We then have the following result.

\begin{lemma}\label{lemma:single_opt_pure}
  Consider the special case where a single campaign $i=1$ gets impressions from a single targeting group $j=1$, and  $I_1$ is the total number of impressions that the campaign is required to obtain.
Furthermore, let the bid value $b^*$ be as given by Eq.~\eqref{eq:opt_b_1}~and~\eqref{eq:opt_b_2}.
Then the pure bidding strategy given by
$$b^*_{1,1} = b^* \mbox{ and } \gamma^*_{1,1} = \frac{I_1}{D_1(b^*)}$$
is an optimal pure bidding strategy.
Furthermore the cost under the optimal pure strategy is given by
$$ I_1 b^* -  \gamma^*_{1,1} \int_0^{b^*} D_1(b) db.$$
\end{lemma}
The result of Lemma~\ref{lemma:single_opt_pure} follow from the fact that the bid value  $b^*$ be as given by Eq.~\eqref{eq:opt_b_1}~and~\eqref{eq:opt_b_2} is the lowest bid value to obtain the $I_1$ impressions required by the campaign. Note that if we have that $D_j(b^*) = I_1$, then the result of Lemma~\ref{lemma:single_opt_pure} is identical to the result of Proposition~\ref{prop:single_opt_cont}.


\subsection{Mixed Bidding Strategies}\label{ssec:single_mixed}
We next consider a more general class of bidding strategies, where a campaign can use several bid values for a given targeting group. We refer to such strategies as mixed bidding strategies.
For a given mixed bidding strategy $s$, let $\setp_{1,1}$ be the bid values that the campaign $1$ uses for targeting group $1$, and let $\gamma_{1,1}(b)$, $b \in \setp_{1,1}$,  be the fraction of bid requests from targeting group $1$ for which the campaign uses bid value $b$, where we have the constraint that
$$\sum_{b \in \setp_{1,1}} \gamma_{1,1}(b) \leq 1.$$
Let $\setS$ be the set of all such mixed bidding strategies.

The cost $C(s)$ of a mixed bidding strategy $s \in \setS$ to obtain
$$\sum_{b \in \setp_{1,1}} \gamma_{1,1}(b) D_1(b)$$
impressions is  given by
$$C(s) = \sum_{b \in \setp_{1,1}} \gamma_{1,1}(b) \left [ D_1(b) b - \int_0^b D_1(x)dx \right ].$$

We have the following result.

\begin{proposition}\label{prop:single_opt_cost}
Consider the special case where a single campaign $i=1$ gets impressions from a single targeting group $j=1$, and $I_1$ is the total number of impressions that the campaign is required to obtain.
Let $s \in \setS$ given  mixed bidding strategy such that
$$\sum_{b \in \setp_{1,1}} \gamma_{1,1}(b) D_1(b) = I_1.$$
Then we have
$$ C(s) \geq I_1 b^* - \int_0^{b^*} D_1(b) db,$$
where the bid value $b^*$ is as given by Eq.~\eqref{eq:opt_b_1}~and~\eqref{eq:opt_b_2}.
\end{proposition}
We provide a proof for Proposition~\ref{prop:single_opt_cost} in Appendix~\ref{app:single_opt_cost}. Proposition~\ref{prop:single_opt_cost} provides a lower-bound on the cost obtained by any mixed bidding strategy.

The next result shows that the bound of Proposition~\ref{prop:single_opt_cont} is tight. 
\begin{proposition}\label{prop:single_opt_alloc}
  Consider the special case where a single campaign $i=1$ gets impressions from a single targeting group $j=1$, and $I_1$ is the total number of impressions that the campaign is required to obtain.
Let the bid value $b^*$ be as given by Eq.~\eqref{eq:opt_b_1}~and~\eqref{eq:opt_b_2}, and suppose that $D^{(-)}(b^*) > 0$.
Furthermore, let the mixed bidding strategy $s(b_1) \in \setS$, $0 \leq b_1 < b^*$, be given by $\setp_1(s) = \{b_1,b^*\}$, 
as well as 
$$\gamma_{1,1}(b_1) = \frac{D_1(b^*) -I_1}{D_1(b^*) - D_1(b_1)} \; \mbox{ and } \;
\gamma_{1,1}(b^*) = 1- \gamma_{1,1}(b_1).$$
Then we have that
$$\lim_{b_1 \uparrow b^*} C(s(b_1)) =  I_1 b^* - \int_0^{b^*} D_1(b) db.$$
\end{proposition}
We provide a proof for Proposition~\ref{prop:single_opt_alloc} in Appendix~\ref{app:single_opt_alloc}. Proposition~\ref{prop:single_opt_alloc} is that it implies that it is sufficient to consider mixed bidding strategies that use two bid values $b_1$ and $b^*$, $b_1 < b^*$, in order to obtain a (near) optimal cost.

\subsection{Cost Bound for Optimal Pure Strategy}
The above results imply that using mixed bidding strategies can lead to a strictly lower overall cost compared with pure strategies. The following result provides a bound on this cost difference.

\begin{proposition}\label{prop:single_cost_bound}
  Consider the special case where a single campaign $i=1$ gets impressions from a single targeting group $j=1$, and $I_1$ is the total number of impressions that the  campaign is required to obtain. Furthermore, let the bid value $b^*$  be as given by Eq.~\eqref{eq:opt_b_1}~and~\eqref{eq:opt_b_2}, and let $s^* \in \setSone$ be the optimal pure bidding strategy as given by Lemma~\ref{lemma:single_opt_pure}. Then we have that
\begin{eqnarray*}
  && C(s^*) -  \left [ I_1 b^* - \int_0^{b^*} D_1(b) db \right ] \\
  &\leq& (1 - \gamma^*_{1,1}) \int_0^{b^*} D_1(b^*) db \\
&\leq& \frac{D_1(b^*) - D^{(-)}_1(b^*)}{D_1(b^*)}  \int_0^{b^*} D_1(b) db,
\end{eqnarray*}
where $  D^{(-)}(b^*) = \lim_{b \uparrow b^*} D_1(b)$.
\end{proposition}
Proposition~\ref{prop:single_cost_bound} follows from  Lemma~\ref{lemma:single_opt_pure}. Proposition~\ref{prop:single_cost_bound} states that the smaller the height (jump) $\Bbl D_1(b^*) - D^{(-)}_1(b^*) \Bbr$ 
of the discontinuity at $b^*$, the closer the performance of the optimal pure strategy will be to the lower-bound of Proposition~\ref{prop:single_opt_cost}.

\subsection{Discussion}
The results for case where a single campaign gets impressions from a single targeting group provide important insight into the structure of an optimal bidding strategy.
In particular, the results suggest that studying optimal policies under the assumption that the supply curve is continuous provides correct structure of an optimal bidding strategy.
More precisely, Proposition~\ref{prop:single_opt_cost} suggests that the optimal cost in the case of continuous supply curves provides the correct expression for the (tight) lower-bound of the cost of any mixed strategy for the case where the supply curves are not necessarily continuous.
In addition, the optimal value $b^*$ obtained in Proposition~\ref{prop:single_opt_cont} corresponds to the optimal bid value of a pure strategy as given in Lemma~\ref{lemma:single_opt_pure}, as well as one of the bid values of the general bid strategy given in Proposition~\ref{prop:single_opt_alloc}.

In the following we show that the results and insights obtained for the case of a single campaign getting impressions from a single targeting group carry over to the general case where multiple campaigns getting impressions from several targeting groups. To do that, we proceed as follows.
We first characterize in Section~\ref{sec:opt_solution} the structure of an optimal allocation $(\opt)$ under the assumption that the supply curves are given by continuous function. Using this result,  we propose in Section~\ref{sec:algorithm} an algorithm to compute an allocation $(\opt)$. In Section~\ref{sec:results} we show that this allocation is an optimal bidding strategy when the supply curves are given by continuous functions. In addition, we derive for this case a lower-bound on the optimal cost of any mixed bidding strategy, as well as mixed bidding strategies that either achieve the lower-bound, or can get arbitrarily close to it.




\section{Structural Properties}\label{sec:opt_solution}

In this section we characterize an optimal  solution $(\opt)$ for the optimization problem given by Eq.~\eqref{eq:opt_problem} under the simplifying assumption that the supply curves $D_j(b)$, $j \in \setT$, are given by continuous functions.


To do this, we use the following notation.
Let $(\opt)$ be a solution to the the optimization given by Eq.~\eqref{eq:opt_problem}. For each campaign $i \in \setN$  let the set $\setA_i(\optbg)$ be given by 
$$\setA_i(\optbg) = \{j \in \setA_i | \gamma^*_{i,j} D_j(b^*_{i,j}) > 0\}.$$
The set $\setA_i(\optbg)$ contains all the targeting groups $j$ from which a campaign receives a positive number  $\gamma^*_{i,j} D_j(b^*_{i,j})$ of impressions under the solution $(\opt)$. 

Similarly, for each targeting group $j \in \setT$ let the set $\setB_j(\optbg)$ be given by
$$\setB_j(\optbg) = \{i \in \setB_j |  \gamma^*_{i,j} D_j(b^*_{i,j}) > 0 \}.$$

When the have the result that there always exists an optimal solution $(\opt)$ for the optimization problem  given by Eq.~\eqref{eq:opt_problem} that consists of components
$$\setCp = (\setCpc, \setCpt)$$
where $\setCpc$ is a set of campaigns, and $\setCpt$ is a set of targeting groups, such that for each campaign $i \in \setCpc$ we have that
$$ b^*_{i,j} = p^*, \qquad j \in \setA_i(\optbg).$$

\begin{proposition}\label{prop:decomposition}
Suppose that the supply curves $D_j(b)$, $j \in \setT$, are given by continuous functions. Then there exists an optimal solution $(\opt)$ to the optimization problem given by Eq.~\eqref{eq:opt_problem} such that the following is true.
There exists a set of bid values $\setP$ such that with every bid value $p^* \in \setP$ we can associate a  component 
$$\setCp = (\setCpc, \setCpt)$$
where $\setCpc$ is a set of campaigns, and $\setCpt$ is a set of targeting groups,  such that
$$ \setA_i(\optbg) \cap \setCpt =  \setA_i(\optbg), \quad i \in \setCpc,$$
and
$$ \setB_j(\optbg) \cap \setCpc =  \setB_j(\optbg), \quad j \in \setCpt,$$
as well as 
$$ b^*_{i,j} = p^*, \qquad i \in \setCpc, j \in  \setA_i(\optbg).$$
In addition, we have that
$$\bigcup\limits_{p^* \in \setP} \setCpc = \setN$$
and
$$\bigcup\limits_{p^* \in \setP} \setCpt = \setT.$$
\end{proposition}


\section{Algorithm to Compute $(\opt)$}\label{sec:algorithm}
Using the results of Section~\ref{sec:opt_solution}, we present in this section an algorithm to compute a  solution $(\opt)$ for the optimization problem given by Eq.~\eqref{eq:opt_problem}. In the next section, we show that the solution $(\opt)$  is an optimal solution for the case where the supply curves are given by continuous functions. In addition, we derive a lower-bound on the optimal cost of any bidding strategy, as well as mixed bidding strategies that either achieve the lower-bound, or can get arbitrarily close to it. 

\subsection{Algorithm}\label{ssec:algo_def}

Using the result of Proposition~\ref{prop:decomposition}, we derive an algorithm that  computes an allocation $(\opt)$ that consists of  components  as given in Proposition~\ref{prop:decomposition}. To do that, the algorithm recursively computes components $\setDp$,
where each component consists of a set of campaigns $\setDpc \subseteq \setN$ and a set of targeting groups $\setDpt \subseteq \setT$. Using these components, we then construct an allocation $(\opt)$ as given in Proposition~\ref{prop:decomposition}.

More precisely, let
$$ \setC_c = \setN \mbox{ and } \setC_t = \setT$$
be the initial set of campaigns and targeting groups, and let
$$\setp = \emptyset.$$
We then split the set $(\setC_c,\setC_t) = (\setN,\setT)$ into two subsets $(\setC^1_c,\setC^1_t)$ and $(\setC^2_c,\setC^2_t)$ as follows.

We first compute the unique bid value $p$ such that
\begin{equation}\label{eq:price1}
  \sum_{j \in \setC_t} D_j(p) \geq \sum_{i \in \setC_c} I_i
\end{equation}  
and
\begin{equation}\label{eq:price2}
  \sum_{j \in \setC_t} D_j(b) < \sum_{i \in \setC_c} I_i, \qquad  0 \leq b <p.
\end{equation}  

Using the bid value $p$, we next compute an optimal solution to the optimization problem $ALLOC(\setC_c,\setC_t,p)$ given by 
\begin{equation}\label{eq:opt_alloc}
\begin{aligned}
& \underset{\alloct}{\text{minimize}}
& & \sum_{i \in \setC_c} \bigg( I_i - \sum_{j \in A_i \cap \setC_t} t_{i,j} D_j(p) \bigg)^2\\
& \text{subject to}
& & \sum_{i \in \setB_j \cap \setC_c} t_{i,j} \leq 1, \; j \in \setC_t, \\
& & & 0 \leq t_{i,j}, \; i \in \setC_c, j \in \setC_t.\\
\end{aligned}
\end{equation}
Note that the optimization problem $ALLOC(\setC_c,\setC_t,p)$ is a convex (quadratic) optimization problem, and can easily been solved using standard algorithms (solvers) of  quadratic programming.

If the minimization problem given by Eq.~\eqref{eq:opt_alloc} has an optimal solution $\optt$ with cost zero, then we stop the recursion, and return the component $\setDp$ where 
$$\alloct = \optt$$
is equal to the optimal solution $\optt$ for the   optimization problem $ALLOC(\setC_c,\setC_t,p)$, and $p$ is the bid value that we used to in the optimization problem $ALLOC(\setC_c,\setC_t,p)$.  In addition, we add the bid value $p$ to the set $\setp$ and let 
$$ \setp = \setp \cup \{p\}.$$

On the other hand, if the  minimization problem given by Eq.~\eqref{eq:opt_alloc} has an optimal solution $\optt$ with a strictly positive cost, then we split the set  $(\setC_c,\setC_t)$ into two sets s $(\setC^1_c,\setC^1_t)$ and $(\setC^2_c,\setC^2_t)$ as follows. We set
\begin{equation}\label{eq:split1}
  \setC_c^{(1)} = \left \{ i \in \setC_c \left | \bigg( I_i - \sum_{j \in \setA_i \cap \setC_t} t^*_{i,j} D_j(p) \bigg)^2 > 0 \right . \right \},
\end{equation}
and
\begin{equation}\label{eq:split2}
  \setC_t^{(1)} = \union_{i \in \setC_c^{(1)}} \setA_i.
\end{equation}
Furthermore, we set
\begin{equation}\label{eq:split3}
  \setC_c^{(2)} =  \setC_c \backslash  \setC_c^{(1)}
\mbox { and } 
 \setC_t^{(2)} =  \setC_t \backslash  \setC_t^{(1)}.
\end{equation}
For each of the two components we then recursively apply the above step until all recursions stop.

More formally we define the algorithm as follows. Let the function
$$p = PRICE(\setC_c,\setC_t)$$
return the unique price $p^*$ as defined by Eq.~\eqref{eq:price1}~and~Eq.~\eqref{eq:price2}.
Furthermore let the function
$$(\optt,done) = ALLOC(\setC_c,\setC_t,p)$$
return an optimal solution to the optimization problem given by Eq~\eqref{eq:opt_alloc},
as well as a  boolean value $done = TRUE$
if the optimal solution $\optt$ for the optimization problem  by Eq~\eqref{eq:opt_alloc} has a cost equal to 0, i.e. we have that
$$\sum_{i \in \setC_c} \bigg( I_i - \sum_{j \in A_i \cap \setC_t} t^*_{i,j} D_j(p) \bigg)^2 = 0,$$
and a boolean value $done = FALSE$ otherwise.
Finally let the function
$$((\setCcone,\setCtone),(\setCctwo,\setCttwo)) = SPLIT(C_c, C_t, t^*, p)$$
return the sets $(\setCcone,\setCtone)$ and  $(\setCcone,\setCtone)$, as given by Eq.~\eqref{eq:split1}~-~\eqref{eq:split3}.

\begin{algorithm}
    \SetKwInOut{Input}{Input}
    \SetKwInOut{Output}{Output}

    $p$ = \textit{PRICE}($C_c, C_t$)\;
    $(\optt,done)$ = \textit{ALLOC}($C_c, C_t, p$)\;

    \eIf{done}{
      $\setC_{out} = \setC_{in} \cup \{(C_c, C_t, t^*,p)\}$\;
      $\setp_{out} = \setp_{in} \cup \{p\})$\;
        return $(\setC_{out}, \setp_{out})$\;
    }{
        $((\setCcone,\setCtone),(\setCctwo,\setCttwo)) = SPLIT(C_c, C_t, t^*, p)$\;
        $(\setC_{out,1}, \setp_{out,1}) = REC(\setCcone, \setCtone, \setC_{in}, \setp_{in})$\;
        $(\setC_{out,2}, \setp_{out,2}) = REC(\setCctwo, \setCttwo, \setC_{in}, \setp_{in})$\;
        $\setC_{out} = \setC_{out,1} \cup  \setC_{out,2}$\;
      $\setp_{out} = \setp_{out,1} \cup \setp_{out,2} $\;
        return $(\setC_{out}, \setp_{out})$\;
    }
\caption{$(\setC_{out},\setp_{out})$ = \textit{REC} $(\setC_c, \setC_t, \setC_{in}, \setp_{in})$ }\label{algo:rec}
\end{algorithm}

 Using these functions, the pseudo-code for the recursion step
$$(\setC_{out},\setp_{out}) = REC(\setC_c, \setC_t, \setC_{in}, \setp_{in})$$
 is then given by Algorithm~\ref{algo:rec}. To simplify the notation, it is assumed that the set $\setA_i$, $i \in \setN$, and $\setB_j$, $j \in \setT$, are available to the function  $REC(\setC_c, \setC_t, \setC_{in}, \setp_{in})$ given in Algorithm~\ref{algo:rec} as global (globally defined) variables.

 Using the recursive function $ REC(\setC_c, \setC_t, \setC_{in}, \setp_{in})$, the algorithm to compute the set $\{\setDp\}_{p \in \setp}$ of components is given by
\begin{equation}\label{eq:algo}
  (\setC_{out},\setp) = REC(\setN,\setT, \emptyset, \emptyset)
\end{equation}
where $\setp$ is the set of bid values of the components $\setDp$, and 
$$ \setC_{out} = \{\setDp\}_{p \in \setp}$$
is the set of components obtained by the algorithm.

\subsection{Construction of Allocation $(\opt)$}\label{ssec:construction_alloc}

Using the set $\setp$ of bid values, and the components $\setDp\}$, $p \in \setp$, obtained by the algorithm given by Eq~\eqref{eq:algo}, we construct a  solution $(\opt)$ as given by Proposition~\ref{prop:decomposition} as follows. Let
$$ \setP = \setp,$$
and for each component  $\setDp$, $p \in \setp$, create a component $\setCp$ with $p^* = p$,
and
$$\setCpc = \setDpc \mbox{ and } \setCpt = \setDpt.$$
Finally using the component $\setDp$ with $p = p^*$, the allocation  for component $\setCp$, $p^* \in \setP$, is given by
$$b^*_{i,j} = p^*, \qquad i \in \setCpc, j \in \setCpt,$$
as well as
$$\gamma^*_{i,j} = t_{i,j}, \qquad i \in \setCpc, j \in \setCpt$$
and
$$\gamma^*_{i,j} = 0, \qquad i \in \setCpc, j \notin \setCpt.$$

\section{Results}\label{sec:results}
In this section we analyze the allocation $(\opt)$ obtained by the algorithm in Section~\ref{sec:algorithm}, and present our main results.  In particular  we show in the next subsection that if the supply curves are given by contiguous function, then the allocation $(\opt)$ obtained by the algorithm is an optimal solution to the optimization problem given by Eq.~\eqref{eq:opt_problem}. In Subsection~\ref{ssec:opt_cost} we use the allocation $(\opt)$ obtained by the algorithm to derive a lower-bound for the cost of any mixed bidding strategy for the the optimization problem  given by Eq.~\eqref{eq:opt_problem}, and derive a sequence of mixed bidding strategies that achieves a cost that is equal, or arbitrarily close, to this lower-bound.

We use the following notation to state our results. Let the  allocation $(\opt)$ with components $\setCp$, $p^* \in \setP$,
be the solution to the  allocation obtained by the algorithm in Section~\ref{sec:algorithm}. Then  for all bid values $p^* \in \setP$
let
$$D_{p^*}(b) = \sum_{j \in \setCpt} D_j(b), \qquad  b \geq 0,$$
and
$$I_{p^*} =  \sum_{i \in \setCpc} I_i.$$

\subsection{Continuous Supply Curves}\label{ssec:opt_cont}
We first analyze the allocation $(\opt)$ obtained by the algorithm presented in Section~\ref{sec:algorithm} for the case where the supply curves are given by continuous functions, and show that in this case  allocation $(\opt)$ is an optimal allocation.

\begin{proposition}\label{prop:opt_cont}
Let $(\opt)$ with components $\setCp$, $p^* \in \setP$, be the allocation obtained by the algorithm in Section~\ref{sec:algorithm}. 
If the supply curves $D_j(b)$ are given by continuous functions, then $(\opt)$ is an optimal solution for the optimization problem given by Eq.~\eqref{eq:opt_problem}, with cost
$$\sum_{p^* \in \setP} \left [ I_{p^*} p^* - \int_0^{p^*} D_{p^*}(b) db \right ].$$
\end{proposition}  
We provide a proof for Proposition~\ref{prop:opt_cont} in Appendix~\ref{app:opt_cont}.

\subsection{General Supply Curves}\label{ssec:opt_cost}
We next consider the general case where the supply curves are not necessarily given by continuous functions. For this case, we  consider mixed strategies in addition to pure strategy. Under a mixed strategy  campaign $i$ can use more than one bid value when bidding for impression of a targeting group $j \in \setA_i$. More precisely, for a mixed strategy $s$ we let $\setp_{i,j}$ be the set of bid values that campaign $i$  uses at targeting group $j$, and let $\gamma_{i,j}(b)$, $b \in \setp_{i,j}(s)$, be the fraction of bid requests at targeting group $j$ for which campaign $i$ uses the bid value $b \in \setp_{i,j}$. Let $\setS$ be the set of all such mixed strategies.

In the following we say that a mixed strategy $s \in \setS$  is a solution to the optimization problem given by Eq.~\ref{eq:opt_problem}, if we have that
$$\sum_{j \in \setA_i} \sum_{b \in \setp_{i,j}} \gamma_{i,j}(b) D_j(b) \geq I_i, \qquad i \in \setN, $$
and
$$\sum_{i \in \setB_j} \sum_{b \in \setp_{i,j}} \gamma_{i,j}(b) \leq 1, \qquad j \in \setT.$$

The cost $C(s)$ of a mixed strategy that is a solution for the optimization problem Eq.~\ref{eq:opt_problem} is given by
$$C(s) =  \sum_{i=1}^{N_c}\left [\sum_{j=1}^{N_r} \sum_{b \in \setp_{i,j}}\gamma_{i,j}(b) \left ( D_j(b)b - \int_0^{b}D_{j}(b)db\right )\right ].$$

We then have the following result.
\begin{proposition}\label{prop:opt_cost}
  Let $s \in \setS$ be a mixed strategy that is a solution to the optimization problem given by Eq.~\ref{eq:opt_problem}.
Then we have that
$$ C(s) \geq \sum_{p^* \in \setP} \left [ I_{p^*} p^* - \int_0^{p^*} D_{p^*}(b) db \right ].$$
\end{proposition}
We provide a proof for Proposition~\ref{prop:opt_cost} in Appendix~\ref{app:opt_cost}. 
Similar to the case of a single campaign and a single targeting group of Section~\ref{sec:single_target}, Proposition~\ref{prop:opt_cost} states that the expression for the lower-bound of the cost of any mixed strategy $s$ is equal to the optimal cost in Proposition~\ref{prop:opt_cont} for the case where the supply curves are given by continuous functions. In addition, Proposition~\ref{prop:opt_cost} that the allocation $(\opt)$ obtained by the algorithm of Section~\ref{sec:algorithm} can be used to compute the lower-bound on the cost of any mixed strategy. This is an important result as it provides an algorithm to compute a lower-bound of the cost of any mixed bidding strategy. We discuss in in more details in Section~\ref{sec:implementation} why this result is important in practice.

We next construct a sequence of mixed strategies that either obtains a cost that is equal to the lower-bound given in Proposition~\ref{prop:opt_cost}, or gets arbitrarily close to the lower-bound. More precisely, the next result constructs for each component $\setCp$, $p^* \in \setP$, of the solution $(\opt)$ obtained by the algorithm in Section~\ref{sec:algorithm} a sequence of mixed bidding strategies under which the cost converges to the lower-bound for this component which is given by
$$I_{p^*} p^*   - \int_0^{p^*} D_{p^*}(b)db.$$

\begin{proposition}\label{prop:opt_alloc}
  Let $(\opt)$ with components $\setCp$, $p^* \in \setP$, be the allocation obtained by the algorithm in Section~\ref{sec:algorithm}.
Then for every component $\setCp$, $p^* \in \setP$, the following is true.
  Let $j \in \setCpt$ be a targeting group that belongs to the component $\setCp$, and for all campaigns $i \in \setB_j(\opt)$ let
  $$\gamma_{i,j} = \frac{ \gamma^*_{i,j}}{\sum_{m \in \setB_j(\opt)} \gamma^*_{m,j}}.$$
  Furthermore, let $s_{p^*}(b_1)$ be a mixed bidding strategy for the component $\setCp$ such that for all campaigns $i \in \setCpc$ and $j \in \setCpt$ we have that
  $$\setp_{i,j} = \{b_1,p^*\}, \qquad j \in \setA_i(\opt),$$
as well as 
  $$\gamma_{i,j}(b_1) = \frac{ \gamma_{i,j} D_1(p^*) -\gamma^*_{i,j} D_1(p^*)}{D_1(p^*) - D_1(b_1)}, \qquad j \in \setA_i(\opt),$$
and 
$$ \gamma_{i,j}(p^*) = \gamma_{i,j} - \gamma_{1,1}(b_1), \qquad j \in \setA_i(\opt).$$
Finally, let  $C_{p^*}(s_{p^*}(b_1))$ be the cost of strategy $s_{p^*}(b_1)$ on component $\setCp$.

If $D^{(-)}_{p^*}(p^*) > 0$, 
then we have that
$$\lim_{b_1 \uparrow b^*} C_{p^*}(s_{p^*}(b_1))  =  I_{p^*} p^*   - \int_0^{p^*} D_{p^*}(b)db.$$
\end{proposition}
We provide a proof for Proposition~\ref{prop:opt_alloc} in Appendix~\ref{app:opt_alloc}. In the proof we show that the the allocation $s_{p^*}(b_1)$ is a feasible solution for component $\setCp$ in the sense that for $ j \in \setCpt$ we have
$$\sum_{i \in \setB_j(\optbg)} \Bsbl \gamma_{i,j}(b_1) +  \gamma_{i,j}(p^*) \Bsbr = 1,$$
and for  $i \in \setCpc$ we have
$$\sum_{j \in \setA_i(\optbg)} \Bsbl \gamma_{i,j}(b_1) D_j(b_1) +  \gamma_{i,j}(p^*) D_j(p^*) \Bsbr = I_i.$$

\subsection{Cost Bound for Pure Strategy $(\opt)$}\label{ssec:cost_bound}
Finally we provide a bound on much we loose by using the pure strategy $(\opt)$ by the algorithm in Section~\ref{sec:algorithm} instead of a mixed strategy.

\begin{corollary}\label{prop:cost_bound}
  Let $(\opt)$ be the the allocation  obtained by the algorithm in Section~\ref{sec:algorithm}. Furthermore, let $C(\opt)$ the cost the allocation $(\opt)$. Then we have that
\begin{eqnarray*}
&&  C(\opt) -  \sum_{p \in \setP} \left [ I_{p^*} p^* - \int_0^{p^*} D_{p^*}(b) db \right ] \\
&\leq& \sum_{p^* \in \setP} \sum_{j \in \setCpt} \frac{D_j(p^*) - D^{(-)}_j(p^*)}{D_j(p^*)}  \int_0^{p^*} D_j(b) db.
\end{eqnarray*}
\end{corollary}  
Similar to Section~\ref{sec:single_target}, Proposition~\ref{prop:cost_bound} states that the smaller the height $\Bbl D_j(p^*) - D^{(-)}_j(p^*)\Bbr$  of a discontinuity at $p^*$ of the supply curve of targeting groups in $\setCpt$, $p^* \in \setP$, the closer the performance of pure strategy $(\opt)$ will be to the lower-bound of  Proposition~\ref{prop:opt_cost}.

\section{Practical Considerations}\label{sec:implementation}
The focus of this paper is to analyze the optimization problem given by Eq.~\eqref{eq:opt_problem}, and derive an algorithm to compute an optimal allocation $(\opt)$. In Section~\ref{sec:online_algo} we discussed how a solution to the  optimization problem given by Eq.~\eqref{eq:opt_problem} can be used to implement an online bidding algorithm. In this section we briefly discuss some additional aspects of how the results obtained in the paper are relevant for understanding, and implementing, bidding algorithms on practical systems.

\subsection{Benchmark}
An important result of our analysis is Proposition~\ref{prop:opt_cost} which provides a tight lower-bound on the cost of any mixed strategy. We note that this lower-bound can be easily computed using the algorithm given in Section~\ref{sec:algorithm}. This lower-bound can be used to benchmark the performance of other online bidding algorithms. In addition, Proposition~\ref{prop:cost_bound} can be used to determine the cost difference of the allocation $(\opt)$ obtained by the algorithm in Section~\ref{sec:algorithm} and any  mixed  bidding strategies.
Note that the allocation $(\opt)$ is a pure bidding strategy that is (slightly) easier to compute and implement in practice. If the cost difference is not too large, then using  allocation $(\opt)$ might be preferred over using a mixed strategy.

\subsection{Structural Properties}\label{ssec:imp_decomposition}
Proposition~\ref{prop:decomposition}, and Proposition~\ref{prop:opt_cont}~and~\ref{prop:opt_alloc}, state that an optimal bidding strategy ``decomposes'' the problem given by Eq.~\eqref{eq:opt_problem} into disjoint components consisting of a set of campaigns and targeting group; an an optimal allocation can be obtained by computing an optimal solution for each component individually. In addition, for each component there exists a single bid value $p^*$ that is used by all campaigns in the component to bid for impressions. This means than rather having to decide on a bid value for each admissible campaign and targeting group pair in a given component, it suffices to compute a single bid value $[^*$ for each component. This result is also important for designing  online algorithms that dynamically update bid values and allocations as bid requests arrive as this result allows to significantly reduce the search space for optimal bid values, i.e. the online algorithm has to tune/learn only a single bid value for each component.  

\subsection{Mixed Bidding Strategies}
An important result of our analysis is to show that in the case (which is the norm in practice) where supply curves are not necessarily given by continuous functions, then pure bidding strategies might not be optimal. Moreover, the result in Proposition~\ref{prop:opt_alloc} states it is sufficient to consider mixed strategies that use with at most 2 bid values for each targeting group, where one if the bid values is equal to the bid valued obtained $p^*$ by the algorithm of Section~\ref{sec:algorithm}, and then second bid value is chosen so that it is close to $p^*$. This implies that (near) optimal mixed bidding strategies are relatively simple to compute and implement.

\subsection{Dynamic Bidding Strategies}
Rather than pre-computing the bid values and allocations at the beginning of a decision period, one can also consider the approach where bid values and allocations are dynamically updated during a decision period. As noted in the previous subsection, the results on the structural property of an optimal allocation can be used in this case to significantly reduce the search space. Analyzing this approach is future research.

\subsection{Estimating Supply Curves}
The online algorithm of Section~\ref{sec:online_algo} assumes that (a good estimate of) the supply curves are available at the beginning of a decision period. A discussion of how to estimate the supply curves is outside the scope of this paper, and this question has been addressed in related work such as~\cite{cui}. The result of the experimental case study in~\cite{cui}, as well as experiments that we carried out, suggest that the supply curves can be estimate quite well using historical data. In particular, using historical data over the past few days, the supply curve for decision periods of a few minutes, or even an hour, can be well estimated (predicted). One reason for this is that each targeting group consists of bid requests that come from several applications.   While the supply curve of each application may have random fluctuations (noise) over a decision period of a few minutes or an hour, the aggregated supply curves tend to ``smooth out'' the noise of individual targeting groups and are easier to estimate/predict.
We also note that small estimation errors can be ``corrected'' over a decision period by using the allocation obtained by the algorithm as a starting point, and dynamically make small adjustments to the bid value and allocations. A detailed analysis and discussion of this approach is outside the scope of this paper, and future research. 

We note that the supply curves used in the optimization problem should evolve from one decision period to another. That is, for each decision period one would use the newest estimates/predictions of the supply curves that is available based on historical data.


\subsection{Curse of Dimensionality}
In the case where a bid request type consists of a small number of  variables/properties such as for example the location, gender, and age of the end user, then targeting groups as given by Definition~\ref{def:targeting_groups} can be easily computed/determined. We note that in current practical applications, where the number of properties the bid request type is typically small (around 3-5 properties). However if the bid request types consists of a large number of properties (more than 20 properties) then the number of targeting groups can become too large to be computed and/or stored. This problem corresponds to the well-known ``curse-of-dimensionality''. Fortunately, in this case it is possible to compute in this case ``approximate'' optimal solution by using ``approximate targeting groups'' that can be efficiently computed. This ``approximate'' optimal solution maintains the structural properties on an optimal solution for the optimization problem given by Eq.~\eqref{eq:opt_problem}. The discussion/derivation of such approximate solutions is outside the scope of this paper. We note that in order to be able to derive such approximate optimal solution, it is crucial to first characterize the exact optimal solution to the  the optimization problem given by Eq.~\eqref{eq:opt_problem}.

\section{Conclusions}\label{sec:conclusions}
We studied the optimal bidding strategies for a DSP that represents several ad campaigns with overlapping targeting criteria. Using the concept of targeting groups and supply curves, we formally modelled the problem as an optimization problem, and characterize optimal bidding strategies.

While the focus of the paper is the  formal analysis of optimal bidding strategies, we discussed in Section~\ref{sec:online_algo} how the solution obtained by the algorithm of Section~\ref{sec:algorithm}  can be used to implement an online bidding algorithm. In addition, we discussed in Section~\ref{sec:implementation}  several aspects of the formal analysis, and the online algorithm of Section~\ref{sec:online_algo}, that are relevant for the implementation of bidding algorithms in practice. In collaboration with a DSP, we have implemented and evaluated the online  algorithm  of Section~\ref{sec:online_algo} using historical data obtain from actual ad exchanges, and actual campaigns that the DSP was representing.  Our experiments showed that the online algorithm (using the allocation $(\opt)$ obtained by the algorithm of Section~\ref{sec:algorithm}) allows the DSP to significantly reduce the cost of getting the required impressions for the ad campaigns it represents.

There are several interesting direction for future research based on the results in this paper. In particular, the question of how to approximate an optimal bidding strategy in the case where the number of targeting groups becomes too large in order to be used on a practical system (curse of dimensionality). In addition, an interesting question is how to tune/dynamically change the bid values and allocations during a decision period in order to account for small changes, or estimation errors, in the supply curves. Preliminary research  suggests that the results obtained in this paper can be used to address these two problems.

\newpage
\bibliographystyle{plain}
\bibliography{sample-bibliography}{}

\appendix
\newpage
\section{Properties of Mixed Strategies}\label{app:mixed_strategies}
In this appendix we derive  properties of mixed strategies that we use in our analysis.

\subsection{Mixed Strategies for A Single Targeting Group}
We first consider the case of Section~\ref{sec:single_target} where a single campaign $i=1$ gets impressions from a single targeting group $j=1$, and $I_1$ is the total number of impressions that the campaign  is required to obtain.
Furthermore, let $s \in \setS$ be a mixed strategy as defined in Section~\ref{ssec:single_mixed}. 

Our first result states that for this case it suffices to consider mixed bidding strategies that obtain exactly the required number of impressions from the targeting group. 

\begin{lemma}\label{lemma:mixed_I_1}
Consider the special case where a single campaign $i=1$ gets impressions from a single targeting group $j=1$, and let $I_1$ be the total number of impressions that the campaign  is required to obtained.  
If $s \in \setS$ is a mixed bidding strategy such that   
$$\sum_{b \in \setp(s)}  \gamma_{1,1}(b) D_1(b) > I_1,$$
then there exists a mixed bidding strategy $s' \in \setS$
with
$$\sum_{b \in \setp'_{1,1}}  \gamma'_{1,1}(b) D_1(b) = I_1,$$
that leads to a strictly lower cost, and  we have that
$$C(s') < C(s).$$
\end{lemma}

This result can been proved by showing that reducing the allocations $\gamma_{1,1}(b)$, $b\in \setp_{1,1}$, of strategy $s$ until the exact number of impressions is obtained will lead to a strictly  lower cost compared with $C(s)$. We omit here a detailed derivation.

The next result provides a condition under which it suffices to consider mixed strategies such that
$$\sum_{b \in \setp}  \gamma_{1,1}(b)  = 1.$$

\begin{lemma}\label{lemma:mixed_1}
Consider the special case where a single campaign $i=1$ gets impressions from a single targeting group $j=1$, and let $I_1$ be the total number of impressions that the campaign is required to obtain.  
Suppose that  for a mixed bidding strategy $s \in \setS$ we have that
$$\sum_{b \in \setp}  \gamma_{1,1}(b)  < 1,$$
and  there exists bid values $b_1, b_2 \in \setp_{1,1}$ such that
$$0 < D_1(b_1) < D_1(b_2).$$
Then there exists a mixed bidding strategy $s' \in \setS$ with
$$\sum_{b \in \setp'_{1,1}}  \gamma'_{1,1}(b) D_1(b) = I_1$$
and
$$\sum_{b \in \setp'_{1,1}}  \gamma'_{1,1}(b)  = 1,$$
that leads to a strictly lower cost, and we have that
$$C(s') < C(s).$$
\end{lemma}

\begin{proof}

  We prove the result by constructing a mixed bidding strategy $s' \in \setS$ that leads to a strictly lower cost than the given strategy $s$

  Let
  $$ \gamma_{1,1}(b_1) D_1(b_1) + \gamma_{1,1}(b_2) D_1(b_2) = I,$$
  i.e. $I$ is the amount of impressions that are obtained by bid values $b_1$ and $b_2$ combined. 

  Let $\gamma > 0$ be such that 
  $$\gamma + \sum_{\in \setp_{1,1} \backslash\{b_1,b_2\}} \gamma_{1,1}(b) = 1.$$
  As by assumption we have that
  $$\sum_{b \in \setp}  \gamma_{1,1}(b)  < 1,$$
  it follows that
  $$\gamma > \gamma_{1,1}(b_1) + \gamma_{1,1}(b_2).$$
  
  Let $\Delta_\gamma$ be such that
  $$ \Delta_\gamma D_1(b_1) + (\gamma -\Delta_\gamma) D_1(b_2) = I.$$
  Note that such a $\Delta_\gamma$ exists as we have that
  $$\gamma_{1,1}(b_1) D_1(b_1) + \gamma_{1,1}(b_2) D_1(b_2) = I,$$
  as well as
  $$0 < D_1(b_1) < D_1(b_2)$$
  and
  $$\gamma_{1,1}(b_1) +  \gamma_{1,1}(b_2) < \gamma.$$
  Furthermore, these properties imply that
    $$ \Delta_\gamma > \gamma_{1,1}(b_1)$$
  and
  $$  (\gamma -\Delta_\gamma) <  \gamma_{1,1}(b_2).$$
 
   Using these definitions, we construct the mixed bidding  strategy $s' \in \setS$ as follows. We set 
  $$\setp'_{1,1} = \setp_{1,1},$$
  and let
   $$\gamma'_{1,1}(b_1)  = \Delta_\gamma,$$
  and
  $$\gamma'_{1,1}(b^*)  = \gamma - \Delta_\gamma.$$ 
  For all $b \in \setp'_{1,1}$, $b \neq b_1,b_2$, we set
  $$ \gamma'_{1,1}(b)  = \gamma_{1,1}(b),$$
  and the new strategy $s'$ differs from strategy $s$ only the the allocations for the bid values $b_1$ and $b_2$.

  The difference in the cost for the two strategies in then given by
  \begin{eqnarray*}
    && C(s) - C(s') \\
    &=& \Bbl \gamma_{1,1}(b_1) - \Delta_\gamma \Bbr
    \left [ D_1(b_1)b_1 - \int_{0}^{b_1} D_1(b)db \right ] \\
    && + \Bbl \gamma_{1,1}(b_1) - \gamma + \Delta_\gamma \Bbr
    \left [ D_1(b_2)b_2 - \int_{0}^{b-2} D_1(b)db \right ] \\
    &=& \gamma_{1,1}(b_1) \left [ D_1(b_1)b_1 - \int_{0}^{b_1} D_1(b)db \right ] \\
    && + \Bbl \gamma_{1,1}(b_2) - \gamma \Bbr
    \left [ D_1(b_2)b_2 - \int_{0}^{b_2} D_1(b)db \right ] \\
    &&  - \Delta_\gamma  \left [ D_1(b_1)b_1 -  D_1(b_2)b_2 - \int_{b_1}^{b_2} D_1(b)db \right ] \\
    &>&  \gamma_{1,1}(b_1) \left [ D_1(b_1)b_1 - \int_{0}^{b_1} D_1(b)db \right ] \\
    && + \Bbl \gamma_{1,1}(b_2) - \gamma \Bbr
    \left [ D_1(b_2)b_2 - \int_{0}^{b_2} D_1(b)db \right ] \\
    && - \Delta_\gamma  \left [ D_1(b_1)b_1 - \int_{0}^{b_1} D_1(b)db \right ].
  \end{eqnarray*}

  Using the fact that
  $$ \Delta_\gamma D_1(b_1) + (\gamma -\Delta_\gamma) D_1(b_2) = I,$$
  we obtain that
  \begin{eqnarray*}
    && C(s) - C(s') \\
    &>& \gamma_{1,1}(b_1) \left [ D_1(b_1)b_1 - \int_{0}^{b_1} D_1(b)db \right ] \\
    && + \gamma_{1,1}(b_2) 
    \left [ D_1(b_2)b_2 - \int_{0}^{b_2} D_1(b)db \right ] \\
    && + \gamma \int_{0}^{b_2} D_1(b)db \\
    && - I + \Delta_\gamma  \int_{0}^{b_1} D_1(b)db.
  \end{eqnarray*}   

  Using the fact that
  $$   \gamma_{1,1}(b_1) D_1(b_1) +   \gamma_{1,1}(b_2)  D_1(b_2) = I_1,$$
  we obtain that
  \begin{eqnarray*}
    && C(s) - C(s') \\
    &>& - \gamma_{1,1}(b_1) \int_{0}^{b_1} D_1(b)db \\
    && - \gamma_{1,1}(b_2)  \int_{0}^{b_2} D_1(b)db  \\
    && + \gamma \int_{0}^{b_2} D_1(b)db \\
    && + \Delta_\gamma  \int_{0}^{b_1} D_1(b)db.
  \end{eqnarray*}   

  Finally, using the fact that
  $$\gamma_{1,1}(b_1) +  \gamma_{1,1}(b_2) < \gamma,$$
  we obtain that
    \begin{eqnarray*}
    && C(s) - C(s') \\
      &>& \Delta_\gamma  \int_{0}^{b_1} D_1(b)db \\
      &>& 0.
  \end{eqnarray*}   
This establishes the result of the proposition.
\end{proof}

The next results is the one of Proposition~\ref{prop:single_opt_cost}.

\begin{proposition}\label{prop:single_mixed_strategies}
Consider the special case where a  single campaign $i=1$ gets impressions from a single targeting group $j=1$, and let $I_1$ be the total number of impressions that the campaign is required to obtain.
Let $C(s)$ be the cost of a given  mixed bidding strategy $s \in \setS$ such that
$$\sum_{b \in \setp} \gamma_{1,1}(b) D_1(b) = I_1.$$
Then we have that
$$ C(s) \geq I_1 b^* - \int_0^{b^*} D_1(b) db,$$
where the bid value $b^*$ is as given by Eq.~\eqref{eq:opt_b_1}~and~\eqref{eq:opt_b_2}.
\end{proposition}

\begin{proof}
  Let $s \in \setS$ be the mixed bidding strategy as given in the statement of the proposition, and let the set $B^*(I_1)$ be the set of all bid values $b \in \setp_{1,1}$ such that
  $$D_1(b) = D_1(b^*) = I_1.$$
  Note that if we have that
  $$B^*(I_1) = \setp_{1,1},$$
  then we have by the definition of the cost function $C(s)$ that
  $$ C(s) = I_1 b^* - \int_0^{b^*} D_1(b) db;$$
  and the result of the proposition follows.  
  Therefore we assume in the following that
  $$B^*(I_1) \neq \setp_{1,1}.$$
  For this case, let  the sets $B^{(-)}$ and $B^{(+)}$ be given by
  $$B^{(-)} = \{ b \in \setp_{1,1} | D_1(b) < I_1\}$$
  and
  $$B^{(+)} = \{ b \in \setp_{1,1} | D_1(b) > I_1\}.$$

  We first consider the case where there exists a bid value $b_1 \in B^{(-)}$ such that
  $$D_1(b_1) > 0.$$
  Using Lemma~\ref{lemma:mixed_I_1}~and~\ref{lemma:mixed_1}, without loss of generality we can assume in this case that
  $$\sum_{b \in \setp} \gamma_{1,1}(b) D_1(b) = I_1$$
  and
  $$\sum_{b \in \setp} \gamma_{1,1})(b)= 1,$$
  as otherwise the bidding strategy $s$ is not an optimal bidding strategy.
  It then follows that
  \begin{eqnarray*}
    && \sum_{b \in B^{(-)}} \gamma(b) D_1(b) + \sum_{b \in B^{(+)}} \gamma(b) D_1(b) \\
    &=& I_1 \left [ \sum_{b \in B^{(-)}} \gamma(b)  + \sum_{b \in B^{(+)}} \gamma(b)  \right ].
  \end{eqnarray*}
  Using these results, we obtain that
  \begin{eqnarray*}
    && C(s) -  b^* I_1 + \int_0^{b^*} D_1(b) db  \\
    &=& \sum_{b \in B^{(-)}} \gamma(b) \Bsbl \Bbl b D_1(b) - b^* I_1 \Bbr 
    + \int_b^{b^*} D_1(b)  db  \Bsbr \\
    &+&  \sum_{b \in B^{(+)}} \gamma(b) \Bsbl \Bbl b D_1(b) - b^* I_1 \Bbr
    - \int_{b^*}^b  D_1(b) db  \Bsbr \\
    &\geq& \sum_{b \in B^{(-)}} \gamma(b) \Bsbl  \Bbl b D_1(b) - b^* I_1 \Bbr
    +  D_1(b)(b^* - b)  \Bsbr \\
    &+&  \sum_{b \in B^{(+)}} \gamma(b) \Bsbl \Bbl b D_1(b) - b^* I_1 \Bbr
    -  D_1(b) (b - b^*)   \Bsbr \\
    &=& \sum_{b \in B^{(-)}} \gamma(b) \Bsbl b^* \Bbl D_1(b) - I_1 \Bbr \Bsbr \\
    &+&  \sum_{b \in B^{(+)}} \gamma(b) \Bsbl b^* \Bbl D_1(b) - I_1 \Bbr  \Bsbr \\
    &=& 0.
  \end{eqnarray*}
  This implies that
  $$ C(s) -  b^* I_1 + \int_0^{b^*} D_1(b) db \geq 0,$$
  and the result of the proposition follows for the case where there exists a bid value $b_1 \in B^{(-)}$ such that
  $$D_1(b_1) > 0.$$

  It remains to consider the case where we have that
  $$D_1(b) = 0, \qquad b \in \setB^{(-)}.$$
  Without loss of generality we can assume in this case that
  $$\setB^{(-)} = \emptyset,$$
  and we have that
  $$D_1(b) \geq I_1, \qquad b \in \setp.$$
  Using Lemma~\ref{lemma:mixed_I_1}, without loss of generality we can furthermore assume in this case that 
  $$\sum_{b \in \setp} \gamma_{1,1}(b) D_1(b) = I_1.$$
  Let
  $$ 0 \leq \Delta_\gamma = 1 - \sum_{b \in  \setp} \gamma_{1,1}(b).$$
  It then follows that
  \begin{eqnarray*}
    && C(s) -  b^* I_1 + \int_0^{b^*} D_1(b) db  \\
    &=& \sum_{b \in  B^{(+)}} \gamma(b) \Bsbl \Bbl b D_1(b) - b^* I_1 \Bbr 
    - \int_{b^*}^{b} D_1(b)  db  \Bsbr \\
    && - \Delta_\gamma  \Bsbl b^* I_1 - \int_0^{b^*}  D_1(b) db  \Bsbr \\
    &\geq&  \sum_{b \in B^{(+)}} \gamma(b) \Bsbl \Bbl b D_1(b) - b^* I_1 \Bbr
    - D_1(b) (b - b^*)  \Bsbr \\
    && - \Delta_\gamma  \Bsbl b^* I_1 - \int_0^{b^*}  D_1(b) db \Bsbr   \\
    &=&  \sum_{b \in B^{(+)}} \gamma(b) (D_1(b) - I_1) b^*   \\
    && - \Delta_\gamma  \Bsbl b^* I_1 - \int_0^{b^*}  D_1(b) db  \Bsbr \\
     &\geq&  \sum_{b \in B^{(+)}} \gamma(b) (D_1(b) - I_1) b^*   \\
    && - \Delta_\gamma  \Bsbl b^* I_1 - \int_0^{b^*}  D_1(b) db  \Bsbr \\
    &\geq&  \sum_{b \in B^{(+)}} \gamma(b) (D_1(b) - I_1) b^*   \\
    && - \Delta_\gamma  b^* I_1  \\
    &=& b^* \sum_{b \in B^{(-)}} \gamma(b) D_1(b)\\
    &&  - I_1 b^*. 
  \end{eqnarray*}
  Using the fact that
  $$\sum_{b \in \setp} \gamma_{1,1}(b) D_1(b) = I_1,$$
  we obtain that
  $$C(s) -  b^* I_1 + \int_0^{b^*} D_1(b) db \geq 0.$$
  The result of the proposition then follows.
  
\end{proof}

\subsection{Mixed Strategies on a Set of Targeting Groups}
Next we consider a set  $\setT = \{1,....,k\}$ of targeting groups, and provide a lower bound on the cost $C(s)$ of any mixed strategy $s \in \setS$ to get a total of $I_t$ impressions from these targeting groups. More precisely, we have the following result.

\begin{proposition}\label{prop:mixed_strategies}
  Let   $\setT = \{1,....,k\}$ be a set of targeting groups, and let
  $$D_t(b) = \sum_{j=1}^k D_j(b), \qquad b \geq 0.$$
Furthermore, let $b_t$ be the bid value such that
$$\sum_{j=1}^k D_j(b_t) \geq I_t$$
and
$$\sum_{j=1}^k D_j(b) < I_t, \qquad 0 \leq b < b_t.$$
Then for any mixed strategy $s \in \setS$ that obtains a total of $I_t$ impressions from the set of targeting groups $\setT$ we have that
$$C(s) \geq I_t b_t - \int_0^{b_t} D_b(b)db.$$
\end{proposition}

\begin{proof}
Let $s \in \setS$ be a mixed strategy that obtains $I_t$ impressions from the set  $\setT = \{1,....,k\}$ of targeting groups, and let
$$I_j(s), \qquad j=1,...,k,$$
be the number of impressions that strategy $s$ obtains from targeting group $j$.

Note that in order to prove the result of the proposition, it suffices to show
that for $n=1,2,...,k$, the cost for the strategy $s$ to get a total of
$$I(n,s) = \sum_{j=1}^n I_j(s)$$
impressions from the first $n$ targeting groups is lower-bounded by
$$ I(n,s) b(n,s) - \int_0^{b(n,s)} \sum_{j=1}^n D_j(b)db,$$
where the bid value $b(n,s)$ is such that
$$\sum_{j=1}^n D_j(b(n,s)) \geq I(n,s)$$
and
$$\sum_{j=1}^n D_j(b) < I(n,s), \qquad 0 \leq b < b(n,s).$$

To prove this result, we proceed as follows.
First we note that by Proposition~\ref{prop:single_mixed_strategies} we have for each targeting group $j=1,...,k,$ that the cost for getting $I_j(s)$ impressions from targeting group $j$ is lower-bounded by
$$ I_j(s) b_j(s) - \int_0^{b_j(s)} D_j(b)db,$$
where the bid value $b_j(s)$ is such that
$$D_j(bs) \geq I_j(s)$$
and
$$D_j(b) < I_j(s), \qquad 0 \leq b < b_j(s).$$
We first show that the result is true for $n=2$. Let $C(2,s)$ be the cost for the strategy to get $I(2,s)$ impressions from the first two targeting groups. Without loss of generality, we can assume that
$$ b_1(s) \leq b_2(s).$$
Note that this implies that
$$ b_1(s) \leq b(2,s) \leq b_2(s).$$

Using Proposition~\ref{prop:single_mixed_strategies} we have that
\begin{eqnarray*}
    && C(2,s) - I(2,s)b(2,s) + \int_0^{b(2,s)} \Bsbl D_1(b) + D_2(b) \Bsbr db\\
    &\geq& I_1(s) b_1(s) - \int_0^{b_1(s)} D_1(b) db \\
    && + I_2(s) b_2(s) -  \int_0^{b_2(s)} D_2(b) db \\
  && - I(2,s) b(2,s) +  \int_0^{b(2,s)} \Bsbl D_1(b) + D_2(b) \Bsbr db
\end{eqnarray*}
It then follows that
\begin{eqnarray*}
    && C(2,s) - I(2,s)b(2,s) + \int_0^{b(2,s)} \Bsbl D_1(b) + D_2(b) \Bsbr db\\
    &\geq&  I_1(s) b_1(s) +  I_2(s) b_2(s)  - (I_1(s) + I_2(s))b(2,s) \\
      && + \int_{b_1(s)}^{b(2,s)} D_1(b) db \\
      && - \int_{b(2,s)}^{b_2(s)} D_2(b) db \\
     &\geq&   - I_1(s) (b(2,s) - b_1(s)) \\
     && +  I_2(s) (b_2(s)  -b(2,s)) \\
      && + D_1(b_1(s))(b(2,s) - b_1(s))\\
    &&  - D^{(-)}_2(b_2(s))(b_2(s) - b(2,s)) \\
    &=& +  \Bsbl D_1(b_1(s))- I_1(s) \Bsbr (b(2,s) - b_1(s)) \\
     && +  \Bsbl I_2(s) - D^{(-)}_2(b_2(s)) \Bsbr (b_2(s)  -b(2,s)) \\
  \end{eqnarray*}
As by assumption we have that
$$ b_1(s) \leq b(2,s) \leq b_2(s)$$
as well as
$$  D_1(b_1(s)) > I_1(s)$$
and
$$D_2(b) < I_2(s), \qquad 0 \leq b < b_2(s).$$
it follows that
$$ I_2(s) \geq  D^{(-)}_2(b_2(s))$$
and
$$C(2,s) \geq I(2,s)b(2,s) - \int_0^{b(2,s)} \Bsbl D_1(b) + D_2(b) \Bsbr.$$
This establishes the result for the case of $n=2$.

We can then prove the result by induction. Assume that the result is true for $n$, $n=1,...,k-1$. For this, let the function $D(n,b)$ be given by
$$D(n,b) = \sum_{j=1}^n D_j(b).$$
The result is then also true for $n+1$ as we have that
    \begin{eqnarray*}
    && C(n+1,s) - I(n+1,s)b(n+1,s) + \int_0^{b(n+1,s)} D(n+1,b) db\\
      &\geq&  I(n,s)b(n,s) + \int_0^{b(n,s)} D(n,b) db\\
    && + I_{n+1}(s) b_{n+1}(s) -  \int_0^{b_{n+1}(s)} D_{n+1}(b) db \\
    && - I(n+1,s)b(n+1,s) + \int_0^{b(n+1,s)} D(n+1,b) db\\
    &=&  I(n,s) b_1(s) +  I_{n+1}(s) b_{n+1}(s)  -  I(n+1,s)b(n+1,s) \\
      && + \int_{b(n,s)}^{b(n+1,s)} D(n,b) db \\
      && - \int_{b(n+1,s)}^{b_{n+1}(s)} D_{n+1}(b) db \\
     &\geq&   - I(n,s) (b(n+1,s) - b_1(n,s)) \\
     && +  I_{n+1}(s) (b_{n+1}(s)  -b(n+1,s)) \\
      && + D(n,b(n,s))(b(n+1,s) - b(n,s))\\
    &&  - D^{(-)}_{n+1}(b_{n+1}(s))(b_{n+1}(s) - b(n,s)) \\
    &=&   \Bsbl D(n,b(n,s))- I(n,s) \Bsbr (b(n+1,s) - b(n,s)) \\
      && +  \Bsbl I_{n+1}(s) - D^{(-)}_{n+1}(b_{n+1}(s)) \Bsbr  (b_{n+1}(s)  -b(n+1,s)).  
    \end{eqnarray*}  
As by assumption we have that
$$ b(n,s) \leq b(n+1,s) \leq b_{n+1}(s),$$
as well
$$D(n,b(n,s)) > I(n,s)$$
and
$$D_{n+1}(b) < I_{n+1}(s), \qquad 0 \leq b < b_{n+1}(s),$$
it follows that
$$ I_{n+1}(s) \geq  D^{(-)}_{n+1}(b_{n+1}(s))$$
and
$$C(n+1,s) \geq I(n+1,s)b(n+1,s) - \int_0^{b(n+1,s)} D(n+1,b) db.$$
This establishes the induction step. The result of the proposition then follows.
\end{proof}


\section{Proofs for Section~\ref{sec:single_target}}\label{app:single_target}
In this appendix we prove  Proposition~\ref{prop:single_opt_cont}-\ref{prop:single_opt_alloc} of Section~\ref{sec:single_target}.

\subsection{Proof of Proposition~\ref{prop:single_opt_cont}}\label{app:single_opt_cont}

We start out by proving Proposition~\ref{prop:single_opt_cont} in a more general setting. We have the following result.

\begin{proposition}
Consider the special case where  a single campaign $i=1$ gets  impressions from a single targeting group $j=1$, and let $I_1$ be the total number of impressions that the campaign  is required to obtain.
If there exists a bid value $b^*$ such that
$$D_1(b^*) = I_1,$$
then an optimal bidding strategy in the set of mixed bidding strategies $\setS$ is a pure bidding strategy given by
$$\gamma^*_{1,1}(b^*) = 1.$$
Furthermore, the optimal cost is given by
$$I_1 b^* - \int_0^{b^*} D_1(b) db.$$
\end{proposition}
Proposition~\ref{prop:single_opt_cont} states that if there exists a bid value $b^*$ such that
$$D_1(b^*) = I_1,$$
then the optimal bidding strategy is the pure strategy given in the statement of the proposition.

\begin{proof}
  By Proposition~\ref{prop:single_mixed_strategies}, the cost of getting $I_1$ impressions from the targeting group 1 is lower-bounded by
  $$I^*b^* - \int_0^{b^*} D_1(b)db.$$
 Note that the pure strategy given in the statment of the proposition achieves this lower-bound, and hence is an optimal allocation. 
\end{proof}

\subsection{Proof of Proposition~\ref{prop:single_opt_cost}}\label{app:single_opt_cost}
Recall the statement of  Proposition~\ref{prop:single_opt_cost}.
\begin{proposition}
Consider the special case where a single campaign $i=1$ gets impressions from a single targeting group $j=1$, and let $I_1$ be the total number of impressions that that campaign  is required to obtain.
Let $C(s)$ be the cost of a given  mixed bidding strategy $s \in \setS$. Then we have that
$$ C(s) \geq I_1 b^* - \int_0^{b^*} D_1(b) db,$$
where the bid value $b^*$ is as given by Eq.~\eqref{eq:opt_b_1}~and~\eqref{eq:opt_b_2}. 
\end{proposition}
This result follows immediately from Proposition~\ref{prop:single_mixed_strategies}.

\subsection{Proof of Proposition~\ref{prop:single_opt_alloc}}\label{app:single_opt_alloc}

Recall the statement of Proposition~\ref{prop:single_opt_alloc}.

\begin{proposition}
  Consider the special case where a single campaign $i=1$ gets impressions from a single targeting group $j=1$, and let $I_1$ be the total number of impressions that the campaign  is required to obtain.
Let the bid value $b^*$ be as given by Eq.~\eqref{eq:opt_b_1}~and~\eqref{eq:opt_b_2} and suppose that
$$D^{(-)}(b^*) > 0.$$
Furthermore, let the mixed bidding strategy $s(b_1) \in \setS$, $0 \leq b_1 < b^*$, be given by
$$\setp_{1,1} = \{b_1,b^*\},$$
and 
$$\gamma_{1,1}(b_1) = \frac{D_1(b^*) -I_1}{D_1(b^*) - D_1(b_1)} \; \mbox{ and } \;
\gamma_{1,1}(b^*) = 1- \gamma_{1,1}(b_1).$$
Then we have that
$$\lim_{b_1 \uparrow b^*} C(s(b_1)) =  I_1 b^* - \int_0^{b^*} D_1(b) db.$$
\end{proposition}

\begin{proof}

 To prove Proposition~\ref{prop:single_opt_alloc} we construct a sequence of bidding strategies  $s^{(n)}$ given as follows. 
  Consider a sequence of bid values $b_n$, $n \geq 1$, such that
  $$b_n < b_{n+1}$$
  and
  $$\lim_{n \to \infty} b_n = b^*.$$
  The strategy  $s^{(n)}$ is then given by
  $$\setp^{(n)}_{1,1} = \{ b_n, b^*\},$$
  as well as
  $$ \gamma^{(n)}(b_n) = \frac{D_1(b^*) - I_1}{D_1(b^*) - D_1(b_n)}$$
  and
  $$\gamma^{(n)}(b^*) = 1 -  \gamma^{(n)}(b).$$
  By the assumptions made in the statement of the proposition we have that
  $$ \gamma^{(n)}(b_n) = \frac{D_1(b^*) - I_1}{D_1(b^*) - D_1(b_n)} < 1,$$
  and therefore
  $$ \gamma^{(n)}(b_n) D_1(b_n) + \gamma^{(n)}(b^*)D_1(b^*) = I_1.$$

  Note that the bidding strategy $s^{(n)}$ with $b_n$ is equal to the bidding strategy $s(b_1 = b_n)$ given in the statement of the proposition.
  
  The cost $C(s^{(n)})$ of strategy $s^{(n)}$ is given by
  \begin{eqnarray*}
    C(s^{(n)}) &=&  \gamma^{(n)}(b_n) \left [  D_1(b_n)b_n - \int_0^{b_n} D_1(b)db \right ]\\
   && + (1 -  \gamma^{(n)}(b_n))  \left [  D_1(b^*)b^* - \int_0^{b^*} D_1(b)db \right ] \\
    &=&  \gamma^{(n)}(b_n) D_1(b_n)b_n + (1 -  \gamma^{(n)}(b_n)) D_1(b^*)b^* \\
    && - \int_0^{b^*} D_1(b)db \\
    && +  \gamma^{(n)}(b_n)  \int_{b_n}^{b^*} D_1(b)db\\
    &=&  \gamma^{(n)}(b_n) D_1(b_n)b^* + (1 -  \gamma^{(n)}(b_n)) D_1(b^*)b^* \\
    && + \gamma^{(n)}(b_n) D_1(b_n) (b_n - b^*) \\
    && - \int_0^{b^*} D_1(b)db \\
    && +  \gamma^{(n)}(b_n)  \int_{b_n}^{b^*} D_1(b)db\\
    &=&  I_1 b^* - \int_0^{b^*} D_1(b)db \\
    && + \gamma^{(n)}(b_n) D_1(b_n) (b_n - b^*) \\
    && +  \gamma^{(n)}(b_n)  \int_{b_n}^{b^*} D_1(b)db.
  \end{eqnarray*}
  We then obtain that
   \begin{eqnarray*} 
     && \lim_{n \to \infty}  C(s^{(n)}) \\
     &=&  I_1b^* - \int_0^{b^*} D_1(b)db \\
     &&  + \lim_{b_n \to b^*} \gamma^{(n)}(b_n) D_1(b_n) (b_n - b^*) \\
    && + \lim_{b_n \to b^*}   \gamma^{(n)}(b_n)  \int_{b_n}^{b^*} D_1(b)db \\
    &=&   I_1b^* - \int_0^{b^*} D_1(b)db.
  \end{eqnarray*}
The result of the proposition then follows.

\end{proof}

\section{Properties of an Optimal Allocation}\label{app:opt_solution}

In this appendix we derive properties of an optimal solution $(\opt)$  for the optimization problem given by Eq.~\eqref{eq:opt_problem} for the case where the supply curves are given by continuous functions, and provide the proofs for  results in Section~\ref{sec:opt_solution}. To help the reader we repeat here the results of Section~\ref{sec:opt_solution} in order to keep the flow of the argument.

We proceed as follows. We first provide in Subsection~\ref{app:opt_bids} and Subsection~\ref{app:decomposition} the  for  results in Section~\ref{sec:opt_solution}. Using these results, we derive in Subsection~\ref{ssec:sufficient_cond} sufficient conditions for an allocation $(\opt)$ to be an optimal solution.

\subsection{Properties of the Bid Values of an Optimal Solution}\label{app:opt_bids}

In this subsection we derive structural properties of the bid values $b^*$ of an optimal solution  for the optimization problem given by Eq.~\eqref{eq:opt_problem} for the case where the supply curves are given by continuous functions.
 For this we use the following notation.
Let $(\opt)$ be a solution to the the optimization given by Eq.~\eqref{eq:opt_problem}. Then for each campaign
$$i \in \setN =  \{1,...,N_c\}$$
the set $\setA_i(\optbg)$ is given by 
$$\setA_i(\optbg) = \{j \in \setA_i | \gamma^*_{i,j} D_j(b^*_{i,j}) > 0\}.$$
Or in other words, the set $\setA_i(\optbg)$ contains all the targeting groups $j$ from which campaign receives a positive number  $\gamma^*_{i,j} D_j(b^*_{i,j})$ of impressions under the solution $(\opt)$. 

Similarly, for each targeting group $j=1,...,N_r$ the set $\setB_j(\optbg)$ is given by
$$\setB_j(\optbg) = \{i \in \setB_j |  \gamma^*_{i,j} D_j(b^*_{i,j}) > 0 \}.$$

The next lemma provides a necessary condition for a solution $(\opt)$  to be an optimal solution for the optimization problem given by Eq.~\eqref{eq:opt_problem}.
\begin{lemma}\label{lemma:meet_demand}
If the supply curves $D_j(b)$, $j\in \setT$, are given by continuous functions, then the following is true.
If $(\opt)$ is an optimal solution to the optimization problem given by Eq.~\eqref{eq:opt_problem}, then we have for each campaign $i=1,...,N_c$ that
$$ \sum_{j \in \setA_i(\optbg)} \gamma^*_{i,j} D_j(b^*_{i,j}) = I_i.$$
\end{lemma}

Lemma~\ref{lemma:meet_demand} states that if  $(\opt)$ is an optimal solution to the optimization problem given by Eq.~\eqref{eq:opt_problem}, then we have for each campaign $i=1,...,N_c$ that the acquired number of impressions for campaign $i$ meets the required number of impressions $I_i$ of campaign $i$.
Note that if this is not the case, then we can lower the cost for campaign $i$ by decreasing the total number of impressions that are acquired (by lowering the allocation variables $\gamma^*_{i,j}$, $j \in \setA_i$).


\begin{proposition}\label{prop:q_j}
If the supply curves $D_j(b)$, $j\in \setT$, are given by continuous functions, then the following is true.
Let $(\opt)$ be an optimal solution to the optimization problem given by Eq.~\eqref{eq:opt_problem}. Then for each targeting group $j=1,...,N_r$ there exists a bid value $q^*_j$ such that
$$ D_j(b^*_{i,j}) = D_j(q^*_j), \qquad i \in \setB_j(\optbg).$$
\end{proposition}

\begin{proof}
  We prove the lemma by contradiction. That is we assume that $(\alloc)$ is an optimal solution to the optimization problem given by Eq.~\eqref{eq:opt_problem}, and there exists a targeting group $j$, $j=1,...,N_r$, and two campaigns $m,n \in \setB_j(\allocbg)$ such that
$$D_j(b_{m,j}) < D_j(b_{n,j}).$$
Note that this implies that
$$b_{m,j} < b_{n,j}.$$
We then construct a new allocation $(\opt)$ that is a feasible solution to the  optimization problem given by Eq.~\eqref{eq:opt_problem}, and leads to a strictly lower overall cost compared with the allocation $\allocbg)$. Note that proving this result suffices to prove the proposition as this leads to a contradiction with the assumption that $(\allocbg)$ is an optimal allocation.

To construct the new allocation $(\opt)$  we use the following definitions. Let 
$$ I_{m,j} = \gamma_{m,j} D_j(b_{m,j})$$
and
$$ I_{n,j} = \gamma_{n,j} D_j(b_{n,j}).$$
Furthermore, let
$$I_t =  I_{m,j} +  I_{n,j},$$
and  let the bid value $q^*_j$ be such that
$$( \gamma_{m,j} +  \gamma_{n,j}) D_j(q^*_j) = I_t.$$
As we have that
$$D_j(b_{m,j}) < D_j(b_{n,j}),$$
a bid value $q^*_j$ with the above property exists and we have that
$$D_j(b_{m,j}) < D_j(q^*_j) < D_j(b_{n,j})$$
and
$$b_{m,j} < q^*_j < b_{n,j}.$$

Using these definitions, let $(\opt)$ be the allocation that is identical to $(\alloc)$, expect that
$$b^*_{m,j} = q^*_j$$
and
$$b^*_{n,j} = q^*_j,$$
as well as
$$  \gamma^*_{m,j} =  \frac{I_{m,j}}{D_j(q^*_j) }$$
and
$$  \gamma^*_{n,j} =  \frac{I_{n,j}}{D_j(q^*_j) }.$$
Compared with the allocation $(\alloc)$, the new allocation $(\opt)$ only affects (changes) the allocations and bid values of campaign $m$ and $n$ at targeting group $j$.

As by construction we have that
$$  (\gamma_{m,j} +  \gamma_{n,j}) D_j(q^*_j)  = I_{m,j} + I_{n,j},$$
we have that
$$\gamma^*_{m,j} + \gamma^*_{n,j} =  \gamma_{m,j} +  \gamma_{n,j}.$$
Furthermore, by construction we have that
$$\gamma^*_{m,j} D_j(q^*_j) = I_{m,j}.$$
Similarly, we have that
$$\gamma^*_{n,j} D_j(q^*_j) =  I_{n,j}.$$
These two results show that under the allocation $(\opt)$ both campaign $m$ and $n$ receive the required number of impressions $I_{m,j}$, and $I_{n,j}$, respectively, from the targeting group $j$. It then follows that  the allocation $(\opt)$ is again a feasible solution to the optimization problem given by Eq.~\eqref{eq:opt_problem}.

Therefore in order to show that the allocation  $(\opt)$ leads to a strictly lower cost compared with the allocation $(\allocbg)$, it suffices to show that 
\begin{eqnarray*}
&& \left [ I_{m,j} b_{m,j} - \gamma_{m,j}\int_0^{b_{m,j}} D_j(b) db \right ] \\ 
&& + \left [ I_{n,j} b_{n,j} -  \gamma_{n,j}\int_0^{b_{n,j}}D_j(b) db \right ]\\ 
&>&  I_t q^*_j  -   (\gamma^*_{m,j} + \gamma^*_{n,j}) \int_0^{q^*_j}D_j(b) db\\ 
&=&  ( I_{m,j} +  I_{n,j}) q^*_j  -   (\gamma_{m,j} + \gamma_{n,j}) \int_0^{q^*_j}D_j(b) db, 
\end{eqnarray*}
or
\begin{eqnarray*}
&& I_{m,j} (b_{m,j} - q^*_j) - \gamma_{m,j}\int_{b_{m,j}}^{q^*_j} D_j(b) db  \\ 
&&  + I_{n,j} (b_{n,j} - q^*_j) -  \gamma_{n,j}\int_{q^*_j}^{b_{n,j}}D_j(b) db \\ 
&>& 0.
\end{eqnarray*}
We have that
\begin{eqnarray*}
&& I_{m,j} (b_{m,j} - q^*_j) - \gamma_{m,j}\int_{b_{m,j}}^{q^*_j} D_j(b) db  \\ 
&&  + I_{n,j} (b_{n,j} - q^*_j) -  \gamma_{n,j}\int_{q^*_j}^{b_{n,j}}D_j(b) db \\ 
&>& I_{m,j} (b_{m,j} - q^*_j) + \gamma_{m,j}D_j(b_{m,j} (q^*_j - b_{m,j})  \\ 
&&  + I_{n,j} (b_{n,j} - q^*_j) -  \gamma_{n,j}D_j(b_{n,j} (b_{n,j} - q^*_j) \\
&=&  I_{m,j} (b_{m,j} - q^*_j) + I_{m,j} (q^*_j - b_{m,j})  \\ 
&&  + I_{n,j} (b_{n,j} - q^*_j) -I_{n,j} (b_{n,j} - q^*_j) \\
&=0.
\end{eqnarray*}
This implies that the allocation $(\opt)$ achieves a strictly lower cost compared to the allocation $(\allocbg)$.  This leads to a contradiction as we assumed that the allocation $(\allocbg)$ is an optimal allocation. The result of the lemma then follows.
\end{proof}

The following corollary follows immediately from Proposition~\ref{prop:q_j}.

\begin{corollary}\label{cor:b_j}
Suppose that the supply curves $D_j(b)$, $j \in \setT$, are given by continuous functions. Then there exists an optimal solution  $\opt$ to the optimization problem given by Eq.~\eqref{eq:opt_problem} such that the following is true. For each targeting group $j=1,...,N_r$ we have that
$$ b^*_{i,j} =  q^*_j, \qquad i \in \setB_j(\optbg).$$
\end{corollary}


\begin{proposition}\label{prop:p_i}
If the supply curves $D_j(b)$, $j\in \setT$, are given by continuous functions, then the following is true.
Let $(\opt)$ be an optimal solution to the optimization problem given by Eq.~\eqref{eq:opt_problem}. Then for each campaigns $i = 1,...,N_c$ we have that
$$ D_m(b^*_{i,m}) =  D_m(b^*_{i,n}) = D_m(p^*_i), \qquad  m,n \in \setA_i(\optbg).$$
\end{proposition}

\begin{proof}
  We prove the lemma by contradiction. That is we assume that $(\alloc)$ is an optimal solution to the optimization problem given by Eq.~\eqref{eq:opt_problem}, and there exists a  campaign $i \in \setN$ and targeting groups  $n,m \in \setA_i(\optbg)$ such that
$$ D_m(b_{i,m}) < D_m(b_{i,n})$$
or
$$ D_m(b_{i,m}) > D_m(b_{i,n}).$$
We then show that there exists an allocation $(\opt)$ such that
$$ b^*_{i,m} =  b^*_{i,n}  = p^*_i$$
which leads to a strictly lower cost compared with the allocation $(\alloc)$. This leads to a contradiction as we assumed that the allocation $(\alloc)$ is an optimal allocation. 

In the following we consider the case where
$$ D_m(b_{i,m}) < D_m(b_{i,n}).$$
The case where
$$ D_m(b_{i,m}) > D_m(b_{i,n}).$$
can be proved using the same argument.

To show the result we use the following definitions. Let 
$$ I_{i,m} = \gamma_{i,m} D_m(b_{i,m})$$
and
$$ I_{i,n} = \gamma_{i,n} D_n(b_{i,n}).$$
Furthermore let  $p^*_i$ be such that
$$\gamma_{i,m} D_m(p^*_i)+ \gamma_{i,n} D_n(p^*_i) = I_{i,m} + I_{i,n} = I_t.$$
Note that such a bid value $p^*_i$ exists. Furhtermore, as we have that
$$ D_m(b_{i,m}) < D_m(b^*_{i,n}),$$
it follows that
$$ D_m(b_{i,m}) < D_m(p^*_j) < D_m(b^*_{i,n})$$
and
$$b_{i,m} < p^*_i < b_{i,n}.$$

Using these definitions, let $(\opt)$ be the allocation that is identical to $(\alloc)$, expect that
$$b^*_{i,m} = p^*_i$$
and
$$b^*_{i,n} = p^*_i.$$
Compared with the allocation $(\alloc)$, the new allocation $(\opt)$ only affects (changes) the bid values that campaign $i$ uses to botain  impressions from targeting group  $m$ and $n$. Furthermore, as we have that 
$$\gamma_{i,m} D_m(p^*_i)+ \gamma_{i,n} D_n(p^*_i) = I_{i,m} + I_{i,n} = I_t,$$
it follows the allocation $(\opt)$ is again a feasible solution to the optimization problem given by Eq.~\eqref{eq:opt_problem}.

Therefore in order to show that the allocation  $(\opt)$  leads to a strictly lower cost compared with the allocation $(\allocbg)$, it suffices to show that 
\begin{eqnarray*}
&& \left [ I_{m,j} b_{m,j} - \gamma_{m,j}\int_0^{b_{m,j}} D_j(b) db \right ] \\
&& + \left [ I_{n,j} b_{n,j} -  \gamma_{n,j}\int_0^{b_{n,j}}D_j(b) db \right ]\\
&& >  I_t p^*_i  -  \int_0^{q^*_j} [\gamma^*_{i,m} D_m(b) + \gamma^*_{i,n} D_n(b) ] db \\
&=&  ( I_{m,j} +  I_{n,j})  p^*_i  -  \int_0^{q^*_j} [\gamma_{i,m} D_m(b) + \gamma_{i,n} D_n(b) ] db
\end{eqnarray*}
or
\begin{eqnarray*}
&& I_{m,j} (b_{m,j} - p^*_i) + \gamma_{m,j}\int_{b_{m,j}}^{p^*_i} D_j(b) db  \\
&& + I_{n,j} (b_{n,j}-p^*_i) -  \gamma_{n,j}\int_{p^*_i}^{b_{n,j}}D_j(b) db \\
&>& 0.
\end{eqnarray*}
We have that
\begin{eqnarray*}
&& I_{m,j} (b_{m,j} - p^*_i) + \gamma_{m,j}\int_{b_{m,j}}^{p^*_i} D_j(b) db  \\
&& + I_{n,j} (b_{n,j}-p^*_i) -  \gamma_{n,j}\int_{p^*_i}^{b_{n,j}}D_j(b) db \\
&>&  I_{m,j} (b_{m,j} - p^*_i) + \gamma_{m,j} D_m(b_{m,j})(p^*_i - b_{m,j})  \\
&& + I_{n,j} (b_{n,j}-p^*_i) -  \gamma_{n,j} D_n(b_{n,j}) (b_{n,j}-p^*_i) \\
&=& 0.
\end{eqnarray*}
This implies that the allocation $(\opt)$ achieves a strictly lower cost compared to the allocation $(\allocbg)$.  This leads to a contradiction as we assumed that the allocation $(\allocbg)$ is an optimal allocation. The result of the lemma then follows.
\end{proof}

The following corollary follows immediately from Proposition~\ref{prop:p_i}.
\begin{corollary}\label{cor:b_i}
Suppose that the supply curves $D_j(b)$, $j \in \setT$, are given by continuous functions. Then there exists an optimal solution  $\opt$ to the optimization problem given by Eq.~\eqref{eq:opt_problem} such that the following is true.
For each campaign  $i=1,...,N_c$ we have that
$$ b^*_{i,j} =  p^*_i, \qquad j \in \setA_i(\optbg).$$
\end{corollary}

\subsection{Decomposition of an Optimal Solution}\label{app:decomposition}
In this appendix, we prove Proposition~\ref{prop:decomposition}. Recall the statement of the proposition.
\begin{proposition}
Suppose that the supply curves $D_j(b)$, $j\in \setT$, are given by continuous functions. Then there exists an optimal solution $(\opt)$ to the optimization problem given by Eq.~\eqref{eq:opt_problem} such that the following is true.
There exists a set of  bid values $\setP$ such that with every bid value $p^* \in \setP$ we can associate component 
$$\setCp = (\setCpc, \setCpt)$$
where $\setCpc$ is a set of campaigns, and $\setCpt$ is a set of targeting groups,  such that
$$ \setA_i(\optbg) \cap \setCpt =  \setA_i(\optbg), \qquad i \in \setCpc,$$
and
$$ \setB_j(\optbg) \cap \setCpc =  \setB_j(\optbg), \qquad j \in \setCpt,$$
as well as 
$$ b^*_{i,j} = p^*, \qquad i \in \setCpc, j \in  \setA_i(\optbg)$$
and
$$ b^*_{i,j} = p^*, \qquad j \in \setCpt, i \in  \setB_j(\optbg),$$
as well as
$$\bigcup\limits_{p^* \in \setP} \setCpc = \setN$$
and
$$\bigcup\limits_{p^* \in \setP} \setCpt = \setT.$$
\end{proposition}

\begin{proof}
Let $(\opt)$ be an optimal  solution to the optimization problem given by Eq.~\eqref{eq:opt_problem} such that for each campaign  $i=1,...,N_c$ we have that
$$ b^*_{i,j} =  p^*_i, \qquad j \in \setA_i(\optbg).$$
Note that such an optimal solution exists by Proposition~\ref{cor:b_i}. 

We then iteratively construct a set $\setP$, and a set of components
$$\setCp = (\setCpc,\setCpt), \qquad p^* \in \setP,$$
as given in the statement of the proposition as follows.

Let
$$\setN^{(0)} = \setN =  \{1,...,N_c\}$$
and
$$\setT^{(0)} = \setT =  \{1,...,N_r\}$$
Furthermore, set 
$$\setP = \emptyset.$$

If the iteration hasn't stopped at iteration $k-1$, $k\geq 1$, then we proceed with  iteration step $k$ as follows. If
$$\setN^{(k-1)} = \emptyset,$$
then we stop the iteration.
Otherwise, we set
$$\setN^{(k)} = \setN^{(k-1)}$$
and
$$\setT^{(k)} = \setT^{(k-1)}.$$
As we have in this case that
$$\setN^{(k)} \neq \emptyset,$$
we can pick  campaign  $i$ from the set  $\setN^{(k)}$. Let the bid value $p^*$ be such that for this campaign $i$ we have that
$$ b^*_{i,j} =  p^*, \qquad j \in \setA_i(\optbg).$$
Note that such a $p^*$ exists by assumption.
We then consider the following two cases.
If we have that
$$p^* \notin \setP,$$
then we set
$$\setP \leftarrow \setP \cup \{p^*\},$$
and create a new set
$$ \setCp =  (\setCpc,\setCpt)$$
where
$$\setCpc = \{ i \}$$
and 
$$\setCpt = \setA_i(\optbg).$$
Finally, we set
$$\setA = \{i\}$$
and repeat the following steps until we have that
$$\setA = \emptyset.$$
\begin{enumerate}
\item Pick a campaign $i \in \setA$ and set
$$\setA \leftarrow \setA\backslash \{i\}$$
$$\setN^{(k)} \leftarrow \setN^{(k)} \backslash\{i\}.$$
\item Update component
$$\setCp =  (\setCpc,\setCpt)$$
of the bid value $p^*$ to
$$\setCpc \leftarrow \setCpc \cup \{i\}$$
and
$$\setCpt \leftarrow \setCpt \cup \setA_i(\optbg).$$
\item Set
$$\setN^{(k)} \leftarrow \setN^{(k)} \backslash\{i\}$$
and
$$\setT^{(k)} \leftarrow \setT^{(k)} \backslash \setA_i(\optbg).$$
\item For all targeting groups $j \in  \setA_i(\optbg)$ set
$$b^*_{m,j} = p^*, \qquad m \in \setB_j(\optbg).$$
\item For each targeting group
$$j \in \setA_i(\optbg)$$
set
$$\setA \leftarrow \setA \cup \left ( \setB_j(\optbg) \cap \setN^{(k)} \right ).$$
\item For all campaigns $m \in \setA$ set
$$b^*_{m,j} = p^*, \qquad j \in \setA_i(\optbg).$$
\end{enumerate}
This completes the iteration step $k$.

We make the following observation regarding the above update rules. By construction (initialization and Step (6)), we have for all targeting group $j \in \setA_i(\optbg)$ in Step (4) that
$$b^*_{i,j} =  p^*.$$
Using Lemma~\ref{prop:q_j}, we then have for each targeting group $j \in \setA_i(\optbg)$ in Step (4) that
$$ D_j(b^*_{m,j}) =  D_j(b^*_{i,j}) =  D_j(p^*), \qquad m \in \setB_j(\optbg).$$
This implies that changing the bid value $b^*_{m,j}$ to $p^*$ does not affect that number of impressions that campaing $m$ obtains from targeting group $j$, nor the cost for obtaining the impressions. As such, the allocation $(opt)$ is still an optimal allocation if we makes this change in the bid value $b^*_{m,j}$.

Similarly, by construction (Step (4)) we have for all campaigns  $m \in \setA$ in Step (6) that there exist a targeting group $n \in \setA_i(\opt)$ such that
$$b^*_{m,n} =  p^*.$$
Using Lemma~\ref{prop:p_i}, we then have for each campaign  $m \in \setA$ in Step (6) that
$$ D_j(b^*_{m,j}) =  D_j(b^*_{m,n}) =  D_j(p^*), \qquad j \in \setA_i(\optbg).$$
This implies that changing the bid value $b^*_{m,j}$ to $p^*$ does not affect that number of impressions that campaing $m$ obtains from targeting group $j$, nor the cost for obtaining the impressions. As such, the allocation $(opt)$ is still an optimal allocation if we makes this change in the bid value $b^*_{m,j}$.

By construction of the set $\setP$ and the components
$$\setCp, \qquad p \in \setP,$$
we have that for each campaign $i=1,...,N_c$, that there exists a unique $p^* \in \setP$ such that
$$i \in \setCpc.$$
Similarly, we have for each targeting group $j=1,...,N_r$, that there exists a unique  $p^* \in \setP$ such that
$$j \in \setCpt.$$
Furthermore, by construction we have that for all components $\setCp$, $p^* \in \setP$, that
$$ \setA_i(\optbg) \cap \setCpt =  \setA_i(\optbg), \qquad i \in \setCpc,$$
and
$$ \setB_j(\optbg) \cap \setCpc =  \setB_j(\optbg), \qquad j \in \setCpt,$$
as well as 
$$ b^*_{i,j} = p^*, \qquad i \in \setCpc, j \in  \setA_i(\optbg)$$
and
$$ b^*_{i,j} = p^*, \qquad j \in \setCpt, i \in  \setB_j(\optbg).$$
This establishes the result of the proposition.
\end{proof}

\subsection{Sufficient Conditions}\label{ssec:sufficient_cond}
We next provide sufficient conditions for an an allocation $(\alloc)$ that has the structure of an optimal solution as given in Proposition~\ref{prop:decomposition} to indeed be an optimal solution.

We proceed as follows. We first consider the special case where we are given an allocation $(\opt)$ as in  Proposition~\ref{prop:decomposition} for which we have that
$$\setP = \{p^*\},$$
the allocation $(\opt)$ consists of a single component $\setCp$. Note that we do not assume that the allocation $(\opt)$ is an optimal solution, but only that the allocation has a structure as the allocation given in  Proposition~\ref{prop:decomposition}.

\begin{proposition}\label{prop:sufficient_cond1}
If the supply curves $D_j(b)$, $j\in \setT$, are given by continuous functions, then the following is true.
Let $(\opt)$ be an allocation such that for all campaigns $i=1,...,N_c$ we have that
$$b^*_{i,j} = p^*, \qquad j \in \setA_i(\optbg),$$
as given in Proposition~\ref{prop:decomposition}, i.e. we have that
$$\setP = \{ p^*\}$$
and
$$\setCpc = \setN.$$
If we have that
$$\sum_{j \in \setA_i(\opt)} \gamma^*_{i,j} D_j(p^*) = I_i, \qquad i \in \setCpc,$$
and
$$\sum_{i \in \setCpc} \sum_{j \in \setA_i(\opt)} D_j(p^*) = \sum_{i \in \setCpc} I_i,$$
then the allocation $(\opt)$ is an optimal solution to the optimization problem given in Eq.~\eqref{eq:opt_problem} with cost
$$ \sum_{i \in \setCpc} I_ip^* - \int_0^{p^*}  \sum_{j \in \setA_i(\opt)} D_j(b)db.$$
\end{proposition}

\begin{proof}
This result can be proved using the same argument as the one given in the proof  for Proposition~\ref{prop:mixed_strategies}.
\end{proof}

Next we consider a more general, but still a special case. In particular we consider an allocation $(\opt)$ as given by  Proposition~\ref{prop:decomposition} such that for $p^*,q^* \in \setP$,
$$q^* < p^*$$
we have that
$$ \setA_i \cap \setCqt = \emptyset, \qquad i \in \setCpc.$$
As it turns out, this case is the case that is relevant for our analysis of the proposed algorithm to compute an optimal allocation $(\opt)$ in Section~\ref{ssec:algo_def}.

Again we do not assume that the allocation $(\opt)$ is an optimal solution, but only that the allocation has a structure as the allocation given in  Proposition~\ref{prop:decomposition}.

\begin{proposition}\label{prop:sufficient_cond_app}
If the supply curves $D_j(b)$, $j\in \setT$, are given by continuous functions, then the following is true.
  Let $(\opt)$ be an allocation as given in Proposition~\ref{prop:decomposition}.
If for all $p^* \in \setP$ we have 
\begin{enumerate}
\item[a)]
$$\sum_{j \in \setA_i(\opt)} \gamma^*_{i,j} D_j(p^*) = I_i, \qquad i \in \setCpc,$$
and
$$\sum_{i \in \setCpc} \sum_{j \in \setA_i(\opt)} D_j(p^*) = \sum_{i \in \setCpc} I_i,$$
\item[b)]
$$\sum_{i \in \setCpc} \sum_{j \in \setA_i \cap \setT_0} D_j(p) = 0, \qquad 0 \leq p < p^*,$$
where
$$ \setT_0 = \left \{ j \in \setT \Big | \sum_{i \in \setB_j} \gamma^*_{i,j} D_j(b^*_{i,j}) = 0 \right \},$$
\item[c)] for $q^* < p^*$, $q^*,p^* \in \setP$, that
$$ \setA_i \cap \setCqt = \emptyset, \qquad i \in \setCpc,$$
\end{enumerate}
then the allocation $(\opt)$ is an optimal solution to the optimization problem given by Eq~\eqref{eq:opt_problem} with cost
$$ \sum_{p^* \in \setP} \left [ \sum_{i \in \setCpc} I_i p^* - \int_0^{p^*}  \sum_{j \in \setA_i(\opt)} D_j(b)db \right ].$$
\end{proposition}

\begin{proof}
We use the result of Proposition~\ref{prop:mixed_strategies} to prove the proposition. More precisely, using Proposition~\ref{prop:mixed_strategies} we have that the cost of any solution to the optimization problem of Eq.~\eqref{eq:opt_problem} that obtains the an amount of
$$  \sum_{i \in \setCpc} I_i$$
impressions from the  component
$$\setCp, \quad p \in \setP,$$
is lower-bounded by
$$ \sum_{i \in \setCpc} I_i p^* - \int_0^{p^*}  \sum_{j \in \setA_i(\opt)} D_j(b)db.$$
As a result, the allocation $(\opt)$ is an optimal allocation for each individual component $\setCp$, $p \in \setP$.

Therefore, in order to proof the result, we have to show for each component 
$$\setCap, \quad p \in \setP,$$
that is not possible reduce the cost for getting the required impressions for the campaigns in this component by shifting some of these impressions to another component. We show this result by induction as follows.

To show that this is indeed the case, we use the following notation. For each component $\setCp$, $p \in \setP$, let
$$I_{p^*} =  \sum_{i \in \setCpc} I_i$$
and
$$D_-{p^*}(b) = \sum_{j \in \setCpt} D_j(b), \qquad 0 \leq b.$$

We first show that  for the component corresponding to the highest bid value $p^* \in \setP$, that we can not reduce the for getting the require $I_{p^*}$ impressions  by shifting some of the  impressions to another component.

We then do an induction step where we assume that it is not possible to reduce the the cost for obtaining the impressions  for the components corresponding to the $n$ highest bid values $p^* \in \setP$ by shifting impressions to another component, and then show that in this case it is also not possible to reduce the the cost for obtaining the impressions  for the components corresponding to the $(n+1)$ highest bid values $p^* \in \setP$ by shifting some of the impressions to another component.

Let $p^*$ be the the highest bid value in the set $\setP$, and suppose that we move an amount of $\Delta_I$ of the impressions $I_{p^*}$ to another component. As we have for $q^* \in \setP$, $q^* < p^*$, that
$$ \setA_i \cap \setCqt = \emptyset, \qquad i \in \setCpc,$$
the only option is to move the $\Delta_I$ impressions to the set $\setT_0$, i.e. the set
$$\setT_0 = \left \{ j \in \setT \Big | \sum_{i \in \setB_j} \gamma^*_{i,j} D_j(b^*_{i,j}) = 0 \right \}.$$
Let
$$\setT_0(p^*) = \setT_0 \cap \cup{i \in \setCpc} A_i,$$
i.e. the set $\setT_0(p)$ contains all the targeting groups $j \in \setT_0$ from which at least one campaign $i \in \setCpc$ can obtain impressions. 
Furthermore, let $D_0(b)$ be given by
$$D_0(b) = \sum_{j \in \setT_0(p^*) } D_j(b).$$
Note by definition that we have that
$$D_0(p^*) = 0.$$
Finally, suppose that there exists a bid value $q$ such that
$$D_q(q) = \Delta_I$$
and
$$D_0(b) < \Delta_I, \qquad 0 \leq b < q.$$

Using these definitions and the fact that
$$ D_0(p) = 0$$
and
$$ D^{(-)}_0(q) \leq \Delta_I,$$
we have that that the cost for getting $\Delta_I$ impressions from the targeting groups in the set $\setT_0(p)$ is at least
  \begin{eqnarray*}
   && \Delta_I q - \int_{p^*}^q D_0(b)db\\
    &\geq&  \Delta_I q - D^{(-)}_0(q)(q-p^*)\\
        &\geq&  \Delta_I q - \Delta_I(q-p^*)\\
    & = &  \Delta_I p^*.
  \end{eqnarray*}
Similarly using the bid value $p'$ such that
$$D_{p^*}(p') = I_{p^*} - \Delta_I,$$
the reduction in the cost for reducing the number of impressions obtained from component $\setCp$ by  $\Delta_I$ is at most
  \begin{eqnarray*}
   && I_{p^*} p^* - \int_0^{p^*} D_p(b)db \\
    && - (I_{p^*} - \Delta_I)p' + \int_0^{p'}D_p(b)db\\
    &=& I(\setCap) (p^* - p') + \Delta_I p' \\
    &&  - \int_{p'}^{p^*} D_p(b)db \\
    &\leq& I(\setCap) (p^* - p') + \Delta_I p' \\
    && - D_p(p')(p^*-p')\\
    &=&  (I_{p^*} - D_p(p')) (p^* - p') + \Delta_I p' \\
    &\leq& \Delta_I  (p^* - p')  + \Delta_I p' \\
    &=&  \Delta_I p^*.
  \end{eqnarray*}
  Therefore, we obtain that it is not possible to reduce the cost for getting $I-{p^*}$ impressions for the campaigns in component $\setCp$ by shifting some of the impressions to targeting groups in the set $\setT_0(p)$. 
   This establishes the result that the allocation $(\opt)$ is an optimal allocation for getting  $I_{p^*}$ impressions for the campaigns in component $\setCp$.

Having established the result for the component with the highest bid value, we do an induction step. That is,  we assume that it is not possible to reduce the the for obtaining the impressions  for the campaigns in the components corresponding to the $n$ highest bid values in $\setP$, and then show that it is not possible to reduce the cost for obtaining the impressions  for the campaigns in the components corresponding to the $(n+1)$ highest bid values in $\setP$.

Let $p^*$ be the $(n+1)$th highest bid value and let $\setCp$ be the corresponding component. Note that in this case, the campaigns $i \in \setCpc$ that belong to this component can only get impressions from targeting groups that belong to the set $\setCpt$, from targeting groups that belong to a component with a higher bid value $q^* \in \setP$, or from the set of targeting groups $\setT_0(p^*)$ given by
$$\setT_0(p^*) = \setT_0 \cap \cup{i \in \setCapc} A_i.$$

As we have by assumption for $D_0(b)$ given by
$$D_0(b) = \sum_{j \in \setT_0(p^*) } D_j(b),$$
that
$$D_0(p) = 0,$$
it follows from the same argument as given for the case of the component with the highest bid value, that  it is  not possible to reduce the cost for obtaining the impressions  for the campaigns in the components corresponding to the $(n+1)$ highest bid values $p \in \setP$ by shifting some impressions from the component $\setCp$ to the set $\setT_0(p^*)$.

It remains to consider the case where we shift some impressions from the component $\setCp$ to a component $\setCq$ with a higher bid value $q^*$.

For this case, we use the following definitons. Let
  $$D_{p^*}(b) = \sum_{j \in \setCpt} D_j(b)$$
  and
  $$D_{q^*}(b) = \sum_{j \in \setCqt} D_j(b).$$
  Furthermore, let $p'$ be such that
  $$D_p(p') = I_{p^*} - \Delta_I,$$
  and let $q'$ be such that
  $$D_{q^*}(q') = I_{q^*} + \Delta_I$$
  Under these definitions, by shifting an amount of $\Delta_I$ impressions from component $\setCp$ to the component $\setCq$, we can reduce the cost by at most an amount given by
  \begin{eqnarray*}
   && I_{p^*} p - \int_0^{p^*} D_p(b)db \\
    && - (I_{p^*} - \Delta_I)p' + \int_0^{p'}D_p(b)db\\
    &=& I_{p^*} (p^* - p') + \Delta_I p' \\
    &&  - \int_{p'}^{p^*} D_p(b)db \\
    &\leq& I_{p^*} (p^* - p') + \Delta_I p' \\
    && - D_p(p')(p^*-p')\\
    &=&  (I_{p^*} - D_p(p')) (p^* - p') + \Delta_I p' \\
    &\leq& \Delta_I  (p^* - p')  + \Delta_I p' \\
    &=&  \Delta_I p^*.
  \end{eqnarray*}
  Similarly,  by shifting an amount of $\Delta_I$ impressions from component $\setCp$ to $\setCq$, we increase the cost by at least an amount given by
    \begin{eqnarray*}
     && - I_{q^*} q^* + \int_0^{q^*} D_q(b)db \\
    && (I_{q^*} + \Delta_I)q' - \int_0^{q'}D_q(b)db\\
    &=& I_{q^*} (q' - q^*) + \Delta_I q' \\
    &&  - \int_{q^*}^{q'} D_q(b)db \\
    &\geq& I_{q^*} (q' - q^*) + \Delta_I q' \\
    && - D^{(-)}_q(q')(q'-q^*)\\
    &=&  (I_{q^*}  + \Delta_I - D^{(-)}_{q^*}(q')) (q' - q^*) +  \Delta_I q^*,\\
    &\geq& \Delta_I q^*. \\
    \end{eqnarray*}
    As we have that
    $$p^* < q^*,$$
    it follows that we can not lower the by shifting some impressions from compoenent $\setCp$ to component  $\setCq$.

This establishes the that the allocation $(\opt)$ is optimal for getting the $I_{p^*}$ impressions for the campaigns in the components $\setCp$ with the $(n+1)$ highest bid values $p^* \in \setP$. This completes the induction step, and the proof of the proposition.
\end{proof}

\section{Properties of Solution $(\opt)$ obtained by Algorithm}
In this appendix we derive properties of the set of components $\setDp$, $p \in \setp$  obtained by the algorithm of Section~\ref{sec:algorithm}.

The following two results that follow directly from the definition of the the algorithm.

\begin{lemma}\label{lemma:algo_cover}
Let $\setDp$, $p \in \setp$ be the set of components obtained through the algorithm given by Eq.~\eqref{eq:algo}. Then we have that 
$$\union_{p \in \setp} \setDpc = \setN$$
and
$$\union_{p \in \setp} \setDpt = \setT.$$
\end{lemma}
Lemma~\ref{lemma:algo_cover} states that each campaign $i \in \setN$, and each targeting group $j \in \setT$, belongs to a component $\setDp$, $p \in \setp$.

\begin{lemma}\label{lemma:algo_Tnull}
  Let $\setDq$, $q \in \setp$ be the set of components obtained through the algorithm given by Eq.~\eqref{eq:algo}, let $p \in \setp$ be the lowest bid value in the set $\setp$, and let $\setDp$ be the corresponding component.
Then we have that
$$\setT_0(\opt) \subseteq \setDpt,$$
where
$$\setT_0(\opt) = \left \{ j \in \setT \Big | \sum_{q \in \setp} \sum_{i \in \setB_j \cap \setDqt} t_{i,j} D_j(q) = 0 \right \}.$$
\end{lemma}

Next we consider one recursion step of the algorithm. That is, we are given a set $\setC_c$ of campaigns,  and a set $\setC_t$ of targeting groups, and use the function $SPLIT(\setC_c,\setC_t,t^*,p)$  described in Subsection~\ref{ssec:algo_def} to produce two sets  $(\setCcone,\setCtone)$ and $(\setCctwo,\setCttwo)$.

\begin{lemma}\label{lemma:algo_properties}
Let $\setC_c \subseteq \setN$ be a set of  campaigns,  and $\setC_t \subseteq \setT$ be a set of targeting groups. Furthermore, suppose that the the function $SPLIT(\setC_c,\setC_t,t^*,p)$ in Algorithm~\ref{algo:rec} returns the sets  $(\setCcone,\setCtone)$ and $(\setCctwo,\setCttwo)$, let $p$ be the bid value used in the function $ALLOC(C_c,C_t,p)$, and let $\optt$ be an optimal solution returned by the function  $ALLOC(C_c,C_t,p)$. Then we have that
$$ \setA_i \cap \setCttwo = \emptyset, \qquad i \in \setCcone,$$
as well as
$$t^*_{i,j} = 0, \qquad i \in \setC_c^2, j \in \setC_c^1,$$
and
$$\sum_{i \in \setB_j \cap \setCcone} t^*_{i,j} = 1, \qquad j \in \setCtone.$$
\end{lemma}
As the minimization problem given by Eq.~\eqref{eq:opt_alloc} is a convex optimization problem. The result of Lemma~\ref{lemma:algo_properties} follows directly from the first order conditions of an optimal solution $\optt$ for the  optimization problem given by Eq.~\eqref{eq:opt_alloc}.

\begin{lemma}
Let $\setC_c \subseteq \setN$ be a set of  campaigns,  and $\setC_t \subseteq \setT$ be a set of targeting groups. Furthermore, suppose that the the function $SPLIT(\setC_c,\setC_t,t^*,p)$ in Algorithm~\ref{algo:rec} returns the sets  $(\setCcone,\setCtone)$ and $(\setCctwo,\setCttwo)$, let $p$ be the bid value used in the function $ALLOC(C_c,C_t,p)$, and let $\optt$ be an optimal solution returned by the function  $ALLOC(C_c,C_t,p)$.

Then for all components $\setDq$ that are obtained from the recursion $REC(\setCctwo,\setCttwo,\emptyset,\emptyset)$, we have that
$$q \leq p.$$
\end{lemma}  

\begin{proof}
Suppose that we obtain $K$ components from the recursion $REC(\setCctwo,\setCttwo,\emptyset,\emptyset)$. We then number the components and create the set $\{ \setDqk\}_{k=1}^K$ such that for a given $k$ we have that
$$ \setA_i \cap \setDqtl = \emptyset, \qquad i \in  \setDqck, l=k+1,...,K.$$
Note that we can create this ordering by Lemma~\ref{lemma:algo_properties}.

Recall that we say that the optimization problem
$$ALLOC(\setC_c,\setC_t,p)$$
has a feasible solution if there exists a solution $\optt$  such that
$$ \sum_{j \in \setA_i \cap \setC_t} t^*_{i,j} D_j(p) = I_i, \qquad i \in \setC_c.$$

In order to show the result of the lemma, we have to show that for all $k=1,...,K$ we have that the optimization problem $ALLOC(\setDqck,\setDqtk,p)$, where $p$ is the bid value given in the statement of the lemma, has a feasible solution. We show this result by induction.

First, we show the result for $k=1$, i.e. we show that optimization problem $ALLOC(\setDqcone,\setDqtone,p)$, where $p$ is the bid value given in the statement of the lemma,  has a feasible solution. By definition of the function $SPLIT(C_c,C_t,\optt,p)$, we have that there exists an  optimal solution $\optt$ to  optimization problem $ALLOC(C_c,C_t,p)$ such that
$$\sum_{j \in C_t \cap \setA_i} t^*_{i,j} D_j(p) \geq I_i, \qquad i \in \setDqcone.$$
Furthermore, by Lemma~\ref{lemma:algo_properties} we have for every the optimal solution $\optt$ to the optimization problem $ALLOC(C_c,C_t,p)$ that
$$t^*_{i,j} = 0,  \qquad i \in \setDqcone, j \notin \setDqtone.$$
Combining the above results, we obtain for the optimal solution $\optt$ to the the optimization problem $ALLOC(C_c,C_t,p)$ that
$$\sum_{j \in \setA_i \cap \setDqtone}  t^*_{i,j} D_j(p) \geq I_i, \qquad i \in  \setDqcone.$$
This implies that the optimization problem $ALLOC(\setDqcone,\setDqtone,p)$ has a feasible solution.

Next we do an induction step. That is, we assume that for the components $\setDql$, $l=1,...,k-1$, we have that the optimization problem $ALLOC(\setDqcl,\setDqtl,p)$ has a feasible solution.
We then show that this is also true for the component $\setDqk$, i.e. we have that the optimization problem $ALLOC(\setDqck,\setDqtk,p)$ has a feasible solution.

By the definition of the algorithm there exists a recursion step $REC(\tilde \setC_c,\tilde \setC_t)$ where
$$\settCc \subseteq \setCctwo$$
and
$$\settCt \subseteq \setCttwo,$$
that returns the sets  $(\settCcone,\settCtone)$ and $(\settCctwo,\settCttwo)$
such that
$$\setDqck \subseteq \settCctwo$$
and
$$\settCcone \subseteq \union_{l=1}^{k-1} \setDqcl.$$
By the induction assumption, we have for $l=1,...,k-1$ that
$$\sum_{j \in \setDqtl} D_j(p) \geq \sum_{i \in \setDqcl} I_i.$$
This implies that for the bid value used for the optimization problem $ALLOC(\tilde \setC_c,\tilde \setC_t,q)$ that was used in the recursion step $REC(\tilde \setC_c,\tilde \setC_t,\emptyset,\emptyset)$
we have that
$$ q < p.$$
Furthermore, by the definition of the algorithm we have that there exists an optimal solution $\optt$ to  optimization problem $ALLOC(\tilde \setC_c,\tilde \setC_t,q)$ such that
$$\sum_{j \in \tilde \setC_t \cap \setA_i} t^*_{i,j} D_j(q) \geq I_i, \qquad i \in \setDqck.$$
By Lemma~\ref{lemma:algo_properties} we have for every  optimal solution $\optt$ to the optimization problem $ALLOC(\tilde \setC_c, \tilde \setC_t,q)$ that
$$t^*_{i,j} = 0,  \qquad i \in \setDqck, j \notin \setDqtk.$$
Combining the above results, we obtain for the  optimal solution $\optt$ to the the optimization problem $ALLOC(\tilde \setC_c,\tilde \setC_t,q)$ that  
$$\sum_{j \in \setA_i \cap \setDqtk}  t^*_{i,j} D_j(q) \geq I_i, \qquad i \in  \setDqck.$$
This implies that the optimization problem
$$ALLOC(\setDqck,\setDqtk,p)$$
has a feasible solution. This completes the induction step, and establishes the result of the lemma.

\end{proof}

\begin{lemma}
Let $\setC_c \subseteq \setN$ be a set of  campaigns,  and $\setC_t \subseteq \setT$ be a set of targeting groups. Furthermore, suppose that the the function $SPLIT(\setC_c,\setC_t,t^*,p)$ in Algorithm~\ref{algo:rec} returns the sets  $(\setCcone,\setCtone)$ and $(\setCctwo,\setCttwo)$, let $p$ be the bid value used in the function $ALLOC(C_c,C_t,p)$, and let $\optt$ be an optimal solution returned by the function  $ALLOC(C_c,C_t,p)$.

Then for all the components $\setDq$ that are obtained from the recursion  $REC(\setCcone,\setCtone,\emptyset,\emptyset)$, we have that
$$q > p.$$
\end{lemma}  

\begin{proof}
Suppose that we obtain $K$ components from the recursion $REC(\setCcone,\setCtone)$. We then number the components and create the set $\{ \setDqk\}_{k=1}^K$ such that for a given $k$  we have that
$$ \setA_i \cap \setDqtl = \emptyset, \qquad i \in  \setDqck, l=1,...,k-1.$$
Note that we can create this ordering by Lemma~\ref{lemma:algo_properties}.

Recall that we say that the optimization problem
$$ALLOC(\setC_c,\setC_t,p)$$
has a feasible solution if there exists a solution $\optt$  such that
$$ \sum_{j \in \setA_i \cap \setC_t} t^*_{i,j} D_j(p) = I_i, \qquad i \in \setC_c.$$

In order to show the result of the lemma, we have to show that the optimization problem $ALLOC(\setDqck,\setDqtk,q)$, $k=1,...,K$, does not have a feasible solution for $ q \leq p$,  where $p$ is the bid value given in the statement of the lemma. We show this result by induction. 

First, we show the result for $k=1$, i.e. we show that optimization problem $ALLOC(\setDqcone,\setDqtone,q)$ does not have a feasible solution for $ q \leq p$, where $p$ is the bid value given in the statement of the lemma.
We prove this case by contradiction.  That is we assume that there exists a  $q$, $0 \leq q <p$, such that the optimization problem $ALLOC(\setDqcone,\setDqtone,q)$ has a feasible solution,  and show that this leads to a contradiction. 

Suppose that there exists an  optimal solution $\alloct$ to the optimization problem $ALLOC(\setDqcone,\setDqtone,q)$ such that
$$\sum_{j \in \setDqtone \cap \setA_i} t_{i,j} D_j(q) \geq I_i, \qquad i \in \setDqcone.$$
By Lemma~\ref{lemma:algo_properties} we have  for every  optimal solution $\optt$ to the optimization problem $ALLOC(\setC_c,\setC_t,p)$ that
$$\sum_{i \in \setDqcone} t^*_{i,j} =1,  \qquad j \in \setDqtone.$$
Combining the above results, we obtain for every optimal solution $\optt$ to the the optimization problem $ALLOC(C_c,C_t,p)$ that
$$\sum_{j \in \setA_i \cap \setDqtone}  t^*_{i,j} D_j(p) \geq I_i, \qquad i \in  \setDqcone,$$
as by assumption we have that
$$q \leq p.$$
However this leads to a contradiction to the fact that
$$ \setDqcone \subseteq \setCcone$$
and
$$ \setDqtone \subseteq \setCtone.$$
This establishes that the optimization problem
$$ALLOC(\setDqcone,\setDqtone,q)$$
does not have a feasible solution for $ q \leq p$

Next we do an induction step. That is, we that for the components $\setDql$, $l=1,...,k-1$, we have that
$$q^{(l)} > p.$$
We then show that this also has to be true for component $\setDqk$, i.e. we have that
$$q^{(k)} > p.$$
We prove this case again by contradiction.  That is we assume that for $q$, $0 \leq q <p$, the optimization problem $ALLOC(\setDqck,\setDqtk,q)$ has a feasible solution,  and show that this leads to a contradiction.

By the definition of the algorithm we have in this case for the optimal solution $\alloct$ to the optimization problem
$$ALLOC(\setDqck,\setDqtk,q)$$
that
$$\sum_{j \in \setDqtk \cap \setA_i} t_{i,j} D_j(q) \geq I_i, \qquad i \in \setDqck.$$
Furthermore, by the definition of the algorithm there exists a  recursion step $REC(\settCc,\settCt)$  where
$$\settCc \subseteq \settCcone$$
and
$$\settCt \subseteq \settCtone,$$
that returns the sets  $(\settCcone,\settCtone)$ and $(\settCctwo,\settCttwo)$
such that
$$\setDqck \subseteq \settCcone$$
and
$$\settCctwo \subseteq \union_{l=1}^{k-1} \setDqcl.$$

By Lemma~\ref{lemma:algo_properties} we have for every  optimal solution $\optt$ to the optimization problem $ALLOC(\setC_c,\setC_t,p)$ that
$$\sum_{i \in \setDqk} t^*_{i,j} =1,  \qquad j \in \setDqtk$$
and
$$t^*_{i,j} = 0, \qquad j \in \setDqtk, i \notin \setDqck.$$
Combining the above results, we obtain for every optimal solution $\optt$ to the the optimization problem $ALLOC(C_c,C_t,p)$ that
$$\sum_{j \in \setA_i \cap \setDqk} t^*_{i,j} D_j(p) \geq I_i, \qquad i \in  \setDqck.$$
as by assumption we have that
$$q \leq p.$$
This leads to a contradiction that
$$ \setDqck \subseteq \settCcone$$
and
$$ \setDqtk \subseteq \settCtone.$$
This implies that the optimization problem
$$ALLOC(\setDqk,\setDqtk,q)$$
does not have a feasible solution for $q \leq p$. This also completes the induction step, and establishes the result of the lemma.
\end{proof}

The following result follows directly from the definition of the algorithm given by Eq.~\eqref{eq:algo}. 
\begin{lemma}
Let $\setDp$ be a component returned by the algorithm given by Eq.~\eqref{eq:algo}, then we have that there exists a feasible solution to the optimization problem $ALLOC(\setDpc,\setDpt,p)$, and we have that
$$ \sum_{j \in \setDpt} D_j(b) < \sum_{i \in \setDpc} I_i, \qquad  0 \leq b <p.$$  
\end{lemma}

\section{Proof of Proposition~\ref{prop:opt_cont}}\label{app:opt_cont}
In this section we prove Proposition~\ref{prop:opt_cont}. For the reader's convenience we repeat the statement of the proposition.

\begin{proposition}
  Let $\{\setDp\}_{p \in \setp}$ be the components that the recursions given by Eq.~\eqref{eq:algo} returned, and let the  allocation $(\opt)$ with components
  $$\setCp, \quad p \in \setP,$$
  be given as follows. We set
$$\setP = \setp.$$
Furthermore, for each bid value $p^* \in \setP$, we set
$$\setCpc = \setDpc$$
and
$$\setCpt = \setDpt.$$
Finally, the allocation $(\opt)$ is such that
$$b^*_{i,j} = p^*, \qquad i \in \setCpc, j \in \setCpt$$
as well as
$$\gamma^*_{i,j} = t_{i,j}, \qquad i \in \setCpc, j \in \setCpt$$
and
$$\gamma^*_{i,j} = 0, \qquad i \in \setCpc, j \notin \setCpt.$$
Then $(\opt)$ is an optimal solution.
\end{proposition}

\begin{proof}
  In order to prove Proposition~\ref{prop:opt_cont} we need to show that the allocation $(\opt)$ given by the components
  $$\setCp, \quad p \in \setP,$$
  defined in the statement of the proposition indeed satisfy the sufficient conditions of Proposition~\ref{prop:sufficient_cond_app}. We prove this result by induction as follows.

  We first show for the lowest bid value $p^* \in \setP$, that the component $\setCp$ satisfies the sufficient conditions given in Proposition~\ref{prop:sufficient_cond_app}.

  We then do an induction step where we assume that the the sufficient conditions given in Proposition~\ref{prop:sufficient_cond_app} are satisfied for the $n$ components with the $n$ lowest bid values in the set $\setP$, and then show that the sufficient conditions are satisfied by the component with the $(n+1)$th lowest bid value in $\setP$. 

Let $p^*$ be the lowest bid value in the set $\setP$. Then by construction, i.e. by the definition of the Algorithm~\ref{algo:rec}, we have that
$$\sum_{j \in \setA_i(\opt)} \gamma^*_{i,j} D_j(p^*) = I_i, \qquad i \in \setCpc.$$
In addition, we have  by the definition of Algorithm~\ref{algo:rec} that
$$ \sum_{j \in \setCpt} D_j(b) < \sum_{i \in \setCpc} I_i, \qquad  0 \leq b <p^*,$$
and
$$\setT_0 \subset \setCpt.$$
This implies that
$$\sum_{j \in \setT_0} D_j(b) = 0,  \qquad  0 \leq b <p^*.$$
It then follows that the sufficient conditions given in Proposition~\ref{prop:sufficient_cond_app} are satisfied for the component $\setCp$ with the lowest bid value $p^*$.

Next we assume that the sufficient conditions given in Proposition~\ref{prop:sufficient_cond_app} are satisfied for the $n$ components with the $n$ lowest bid values in the set $\setP$, and then show that the sufficient conditions are satisfied by the component with the $(n+1)$the lowest bid value in $\setP$. More precisely, let $p^* \in \setP$ be the $(n+1)$th lowest bid value, and let $\setCp$ be the corresponding component. 

By Lemma~\ref{lemma:algo_properties}, for all bid values $q^* \in \setP$, such that
$$ q^* < p^*,$$
we have that
$$ \setA_i \cap \setCqt = \emptyset, \qquad i \in \setCpc.$$
In this case, the sufficient conditions given in Proposition~\ref{prop:sufficient_cond_app} are that we have that
$$\sum_{j \in \setA_i(\opt)} \gamma^*_{i,j} D_j(p^*) = I_i, \qquad i \in \setCpc,$$
and
$$ \sum_{j \in \setCpt} D_j(b) < \sum_{i \in \setCpc} I_i, \qquad  0 \leq b <p^*.$$
Note that this is the case by the definition of Algorithm~\ref{algo:rec}.  This completes the induction step, and the proof of the proposition.
\end{proof}

\section{Proof of Proposition~\ref{prop:opt_cost}~and~\ref{prop:opt_alloc}}\label{app:opt_alloc}\label{app:opt_cost}
In this appendix we proof Proposition~\ref{prop:opt_cost} and Proposition~\ref{prop:opt_alloc} of Section~\ref{sec:results}. Before we do that, we derive in the next subsections a bound on the cost of mixed strategies.

Recall the statement of Proposition~\ref{prop:opt_cost}.

\begin{proposition}
  Let $s \in \setS$ be a mixed strategy that is a solution to the optimization problem given by Eq.~\ref{eq:opt_problem}.
Then we have that
$$ C(s) \geq \sum_{p^* \in \setP} \left [ I_{p^*} p^* - \int_0^{p^*} D_{p^*}(b) db \right ].$$
\end{proposition}

\begin{proof}
We use the result of Proposition~\ref{prop:mixed_strategies} to prove Proposition~\ref{prop:opt_cost}. To do that, we use the following notation.  For each component $\setCp$, $p \in \setP$, of the allocation $(\opt)$ obtained by the algorithm in Section~\ref{sec:algorithm}, let
$$I_{p^*} =  \sum_{i \in \setCpc} I_i$$
and
$$D_-{p^*}(b) = \sum_{j \in \setCpt} D_j(b), \qquad 0 \leq b.$$
Using Proposition~\ref{prop:mixed_strategies} we have that the cost of the mixed allocation $(\opt)$ that obtains the an amount of
$$I(\setCp)$$
impressions from component
$$\setCp, \quad p \in \setP,$$
is lower-bounded by
\begin{equation}\label{eq:bound_comp}
I_{p^*} p^* - \int_0^{p^*} D_{p^*}(b)db.
\end{equation}

Therefore, in order to proof the result, we have to show for each component 
$$\setCp, \quad p^* \in \setP,$$
that is not possible reduce the bound given by Eq.~\eqref{eq:bound_comp} by shifting some of the impressions $I_{p^*}$ to another component. We show this result by induction as follows.
  
We first show that  for the component corresponding to the highest bid value $p \in \setP$, that we can not reduce the bound  given by Eq.~\eqref{eq:bound_comp} by shifting some of the  impressions $I_{p^*}$ to another component.

We then do an induction step where we assume that it is not possible to reduce the the bound given by Eq.~\eqref{eq:bound_comp} for obtaining the impressions  for the components corresponding to the $n$ highest bid values $p^* \in \setP$, and then show that it is not possible to reduce the the bound given by Eq.~\eqref{eq:bound_comp}  for obtaining the impressions  for the components corresponding to the $(n+1)$ highest bid values $p^* \in \setP$.

Let $p^*$ be the the highest bid value in the set $\setP$, and suppose that we move an amount of $\Delta_I$ of the impressions $I_{p^*}$ to another component. Note that as the solution satisfies  the sufficient conditions given in Proposition~\ref{prop:sufficient_cond_app}, the only option is to move the $\Delta_I$ impressions to the set $\setT_0$, i.e. the set
$$\setT_0 = \left \{ j \in \setT \Big | \sum_{i \in \setB_j} \gamma^*_{i,j} D_j(b^*_{i,j}) = 0 \right \}.$$
Let
$$\setT_0(p^*) = \setT_0 \cap \cup{i \in \setCpc} A_i,$$
i.e. the set $\setT_0(p^*)$ contains all the targeting groups $j \in \setT_0$ from which at least one campaign $i \in \setCpc$ can obtain impressions. 
Furthermore, let $D_0(b)$ be given by
$$D_0(b) = \sum_{j \in \setT_0(p^*) } D_j(b).$$
Note by definition that we have that
$$D_0(p^*) = 0.$$
Finally, suppose that there exists a bid value $q$ such that
$$D_0(q) \geq \Delta_I$$
and
$$D_0(b) < \Delta_I, \qquad 0 \leq b < q.$$

Using these definitions and the fact that
$$ D_0(p^*) = 0$$
and
$$ D^{(-)}_0(q) \leq \Delta_I,$$
we have that that the cost for getting $\Delta_I$ impressions from the targeting groups in the set $\setT_0(p^*)$ is at least
  \begin{eqnarray*}
   && \Delta_I q - \int_{p^*}^q D_0(b)db\\
    &\geq&  \Delta_I q - D^{(-)}_0(q)(q-p^*)\\
        &\geq&  \Delta_I q - \Delta_I(q-p^*)\\
    & = &  \Delta_I p^*.
  \end{eqnarray*}
Similarly using the bid value $p'$ such that
$$D_{p^*}(p') \geq I_{p^*} - \Delta_I$$
and
$$D_{p^*}(p') < I_{p^*} - \Delta_I, \qquad b < p',$$
the reduction in the lower bound for reducing the number of impressions obtained from component $\setCp$ by  $\Delta_I$ is at most
  \begin{eqnarray*}
   && I_{p^*} p^* - \int_0^{p^*} D_p(b)db \\
    && - (I_{p^*} - \Delta_I)p' + \int_0^{p'}D_p(b)db\\
    &=& I(\setCap) (p^* - p') + \Delta_I p' \\
    &&  - \int_{p'}^{p^*} D_p(b)db \\
    &\leq& I(\setCap) (p^* - p') + \Delta_I p' \\
    && - D_p(p')(p^*-p')\\
    &=&  (I_{p^*} - D_p(p')) (p^* - p') + \Delta_I p' \\
    &\leq& \Delta_I  (p^* - p')  + \Delta_I p' \\
    &=&  \Delta_I p^*.
  \end{eqnarray*}
  Therefore, we obtain that it is not possible to reduce the the bound given by Eq.~\eqref{eq:bound_comp} by shifting some of the impressions  $I_{p^*}$ to targeting groups in the set $\setT_0(p^*)$. 
   This establishes the result that the cost for getting $I_{p^*}$ impressions for the campaigns in component $\setCp$ is bounded from below by
    $$ I_{p^*} p^* - \int_0^{p^*} D_{p^*}(b)db.$$
    This establishes the result for the component $\setCp$ with the highest bid value $p^*$.

Having established the result for the component with the highest bid value, we do an induction step. That is, we assume that it is not possible to reduce the the bound given by Eq.~\eqref{eq:bound_comp} for obtaining the impressions  for the components corresponding to the $n$ highest bid values in $\setP$, and then show that in this case it is not possible to reduce the the bound given by Eq.~\eqref{eq:bound_comp}  for obtaining the impressions  for the components corresponding to the $(n+1)$ highest bid values in $\setP$.

Let $p6*$ be the $(n+1)$th highest bid value and let $\setCp$ be the corresponding component. Note that in this case, the campaigns $i \in \setCpc$ that belong to this component can only get impressions from targeting groups that belong to the set $\setCpt$, from targeting groups that belong to a component with a higher bid value $q^* \in \setP$, or from the set of targeting groups $\setT_0(p^*)$ given by
$$\setT_0(p^*) = \setT_0 \cap \cup{i \in \setCpc} A_i,$$
i.e. the set $\setT_0(p^*)$ contains all the targeting groups $j \in \setT_0$ from which at least one campaign $i \in \setCpc$ can obtain impressions. 

As we have that for $D_0(b)$ given by
$$D_0(b) = \sum_{j \in \setT_0(p^*) } D_j(b),$$
that
$$D_0(p^*) = 0,$$
it follows from the same argument as given for the case of the component with the highest bid value, that  it is  not possible to reduce the the bound given by Eq.~\eqref{eq:bound_comp}  for obtaining the impressions  for the components corresponding to the $(n+1)$ highest bid values in $\setP$ by shifting some impressions from the component $\setCp$ to the set $\setT_0(p^*)$.

There it remains to consider the case where we shift some impressions from the component $\setCp$ to a component $\setCq$ with a higher bid value $q^*$.

  For this case, let $p'$ be such that
  $$D_{p^*}(p') \geq I_{p^*} - \Delta_I$$
  and
  $$D_{p^*}(p') < I_{p^*} - \Delta_I, \qquad 0 \leq b < p',$$
  and let $q'$ be such that
  $$D_{q^*}(q') \geq I_{q^*} + \Delta_I$$
  and
  $$D_{q^*}(q') < I_{q^*} + \Delta_I, \qquad 0 \leq b < q'.$$
  Under these definitions, by shifting an amount of $\Delta_I$ impressions from component $\setCp$ to $\setCq$, we can reduce the bound given by Eq.~\eqref{eq:bound_comp} by at most an amount given by
  \begin{eqnarray*}
   && I_{p^*} p - \int_0^{p^*} D_p(b)db \\
    && - (I_{p^*} - \Delta_I)p' + \int_0^{p'}D_p(b)db\\
    &=& I_{p^*} (p^* - p') + \Delta_I p' \\
    &&  - \int_{p'}^{p^*} D_p(b)db \\
    &\leq& I_{p^*} (p^* - p') + \Delta_I p' \\
    && - D_p(p')(p^*-p')\\
    &=&  (I_{p^*} - D_p(p')) (p^* - p') + \Delta_I p' \\
    &\leq& \Delta_I  (p^* - p')  + \Delta_I p' \\
    &=&  \Delta_I p^*.
  \end{eqnarray*}
  Similarly,  by shifting an amount of $\Delta_I$ impressions from component $\setCp$ to $\setCq$, we increase the lower bound by at least an amount given by
    \begin{eqnarray*}
     && - I_{q^*} q^* + \int_0^{q^*} D_q(b)db \\
    && (I_{q^*} + \Delta_I)q' - \int_0^{q'}D_q(b)db\\
    &=& I_{q^*} (q' - q^*) + \Delta_I q' \\
    &&  - \int_{q^*}^{q'} D_q(b)db \\
    &\geq& I_{q^*} (q' - q^*) + \Delta_I q' \\
    && - D^{(-)}_q(q')(q'-q^*)\\
    &=&  (I_{q^*}  + \Delta_I - D^{(-)}_{q^*}(q')) (q' - q^*) +  \Delta_I q^*,\\
    &\geq& \Delta_I q^*. \\
    \end{eqnarray*}
    As we have that
    $$p^* < q^*,$$
    it follows that we can not lower the bound for getting the impressions for component $\setCp$ by shifting some impressions to component  $\setCq$.

This establishes the result that  the bound given by Eq.~\eqref{eq:bound_comp} provides a lower bound  for obtaining the impressions  for the campaigns in the components with the $(n+1)$ highest bid values $p \in \setP$. This completes the induction step, and the proof of the proposition.

\end{proof}


Next we prove Proposition~\ref{prop:opt_alloc}. Recall the statement of the proposition.

\begin{proposition}
  Let $(\opt)$ with components $\setCp$, $p^* \in \setP$, be the allocation obtained by the algorithm in Section~\ref{sec:algorithm}.
Then for every component $\setCp$, $p^* \in \setP$, the following is true.
  Let $j \in \setCpt$ be a targeting group that belongs to the component $\setCp$, and for all campaigns $i \in \setB_j(\opt)$ let
  $$\gamma_{i,j} = \frac{ \gamma^*_{i,j}}{\sum_{m \in \setB_j(\opt)} \gamma^*_{m,j}}$$
  and
  $$I_{i,j} = \gamma^*_{i,j} D_j(p^*).$$
  Furthermore, let $s_{p^*}(b_1)$ be a mixed bidding strategy for the component $\setCp$ such that for all campaigns $i \in \setCpc$ and $j \in \setCpt$ we have that
  $$\setp_{i,j} = \{b_1,p^*\}, \qquad j \in \setA_i(\opt),$$
as well as 
  $$\gamma_{i,j}(b_1) = \frac{ \gamma_{i,j} D_1(p^*) -I_{i,j}}{D_1(p^*) - D_1(b_1)}, \qquad j \in \setA_i(\opt),$$
and 
$$ \gamma_{i,j}(p^*) = \gamma_{i,j} - \gamma_{1,1}(b_1), \qquad j \in \setA_i(\opt).$$
Finally, let  $C_{p^*}(s_{p^*}(b_1))$ be the cost of strategy $s_{p^*}(b_1)$ on component $\setCp$.

If $D^{(-)}_{p^*}(p^*) > 0$, 
then we have that
$$\lim_{b_1 \uparrow b^*} C_{p^*}(s_{p^*}(b_1))  =  I_{p^*} p^*   - \int_0^{p^*} D_{p^*}(b)db.$$
\end{proposition}

\begin{proof}
To prove the proposition, we proceed as follows.
  Note that by definition we have for every targeting group $j  \in \setCpt$ that
  $$\sum_{i \in \setB_j(\opt)} \Bsbl\gamma_{i,j}(b_1) +  \gamma_{i,j}(b^*) \Bsbr = \sum_{i \in \setB_j(\opt)} \gamma_{i,j} = 1.$$

   Furthermore, for every targeting group $j  \in \setCpt$ we have for every campaign $i \in \setB_j(\opt)$,
   \begin{eqnarray*}
     && \gamma_{i,j}(b_1) D_j(b_1) + \gamma_{i,j}(b^*)D_j(p^*) \\
     &=&  \gamma_{i,j}(b_1) D_j(b_1) 
     +  \Bsbl \gamma_{i,j} - \gamma_{1,1}(b_1) \Bsbr D_j(p^*) \\
     &=&  \gamma_{i,j}  D_j(p^*) -  \gamma_{i,j}(b_1) ( D_j(p^*) -  D_j(b_1) )\\
     &=&  \gamma_{i,j}  D_j(p^*) -  \gamma_{i,j} D_j(p^*) + I_{i,j} \\
     &=&  \gamma^*_{i,j} D_j(p^*) = I_{i,j}.
     \end{eqnarray*}

   Using this result, it follows that for each campaign $i \in  \setCpc$ we have that
\begin{eqnarray*}
  && \sum_{j \in \setA_i(\opt)} \Bsbl\gamma_{i,j}(b_1)  D_j(b_1) +  \gamma_{i,j}(b^*) D_j(p^*) \Bsbr \\
  &=& \sum_{j \in \setA_i(\opt)} \gamma^*_{i,j} D_j(p^*) \\
  &=& I_i.
\end{eqnarray*}  
This means that the mixed strategy $s_{p^*}(b_1)$ obtains for each campaign $i \in \setCpc$ the required number of impressions.

Finally,  under the bidding strategy $s(b_1)$ we have for each campaign $i \in  \setCpc$ the cost for obtaining the
$$I_{i,j} =  \gamma^*_{i,j} D_j(p^*)$$
impressions from targeting group $j \in \setA_i(\opt)$ is given by
\begin{eqnarray*}
 && \gamma_{i,j}(b_1) \left [  D_j(b_1)b_1 - \int_0^{b_1} D_j(b)db \right ]\\
   && + (\gamma_{i,j} -  \gamma_{i,j}(b_1))  \left [  D_j(p^*)p^* - \int_0^{p^*} D_j(b)db \right ] \\
    &=&  \gamma_{i,j}(b_1) D_1(b_1)b_1 + (\gamma_{i,j} -  \gamma_{i,j}(b_1)) D_1(p^*)p^* \\
    && - \gamma_{i,j} \int_0^{p^*} D_j(b)db \\
  && +  \gamma_{i,j}(b_1)  \int_{b_1}^{p^*} D_j(b)db\\
  &=&  \gamma_{i,j}(b_1) D_1(b_1)p^* + (\gamma_{i,j} -  \gamma_{i,j}(b_1)) D_1(p^*)p^* \\
   && +  \gamma_{i,j}(b_1) D_1(b_1)(b_1 - p^*) \\
    && - \gamma_{i,j} \int_0^{p^*} D_j(b)db \\
  && +  \gamma_{i,j}(b_1)  \int_{b_1}^{p^*} D_j(b)db\\
  &=&  I_{i,j} - \gamma_{i,j} \int_0^{p^*} D_j(b)db \\
     && +  \gamma_{i,j}(b_1) D_1(b_1)(b_1 - p^*) \\
    && +  \gamma_{i,j}(b_1)  \int_{b_1}^{p^*} D_j(b)db.
  \end{eqnarray*}
Combining the above results, it then follows that
   \begin{eqnarray*} 
     && \lim_{b_1 \uparrow p^*} C_{p^*}(s_{p^*}(b_1)) \\
     &=&  \sum_{i \in \setCpc} \sum_{j \in \setA_i(\opt)} \left [I_{i,j}p^*   -  \gamma_{i,j} \int_0^{p^*} D_j(b)db \right ]\\
     && +   \lim_{b_1 \uparrow p^*} \sum_{i \in \setCpc} \sum_{j \in \setA_i(\opt)} \gamma_{i,j}(b_1) D_1(b_1)(b_1 - p^*) \\
     && +    \lim_{b_1 \uparrow p^*} \sum_{i \in \setCpc} \gamma_{i,j}(b_1)  \int_{b_1}^{p^*} D_j(b)db \\
     &=&   \sum_{i \in \setCpc} \sum_{j \in \setA_i(\opt)} \left [I_{i,j}p^*   -  \gamma_{i,j} \int_0^{p^*} D_j(b)db \right ]\\
   \end{eqnarray*}
   
     It then follows that     
    \begin{eqnarray*} 
     && \lim_{b_1 \uparrow p^*} C_{p^*}(s_{p^*}(b_1)) \\    
     &=&  \sum_{i \in \setCpc} \sum_{j \in \setA_i(\opt)} I_{i,j}p^* \\
      && - \sum_{i \in \setCpc} \sum_{j \in \setA_i(\opt)} \gamma_{i,j} \int_0^{p^*} D_j(b)db
    \end{eqnarray*}
    Using the fact that
    \begin{eqnarray*} 
      &&  \sum_{i \in \setCpc} \sum_{j \in \setA_i(\opt)} \gamma_{i,j} \int_0^{p^*} D_j(b)db    \\
      &=& \sum_{j \in \setCpt} \sum_{i \in \setB_j(\opt)} \gamma_{i,j} \int_0^{p^*} D_j(b)db, 
    \end{eqnarray*}
we obtain that
    \begin{eqnarray*}
     &=& I_{p^*} p^* 
      - \sum_{j \in \setCpt} \sum_{i \in \setB_j(\opt)} \gamma_{i,j} \int_0^{p^*} D_j(b)db \\
     &=& I_{p^*} p^* 
      - \int_0^{p^*} \left ( \sum_{j \in \setCpt} D_j(b) \right ) db \\
    &=& I_{p^*} p^* - \int_0^{p^*} D_{p^*}(b)db.
\end{eqnarray*}  
The result of the proposition then follows.

\end{proof}

\end{document}